\documentclass{amsart}

\usepackage{amsmath}
\usepackage{amssymb}
\usepackage{commath}
\usepackage{mathtools}
\usepackage[all]{xy}

\usepackage{hyperref}

\theoremstyle{plain}
\newtheorem*{theorem*}{Theorem}
\newtheorem{theorem}{Theorem}
\newtheorem{lemma}[theorem]{Lemma}
\newtheorem{proposition}[theorem]{Proposition}
\newtheorem{corollary}[theorem]{Corollary}
\theoremstyle{definition}
\newtheorem{definition}[theorem]{Definition}
\newtheorem{remark}[theorem]{Remark}

\numberwithin{theorem}{subsection}
\numberwithin{equation}{subsection}

\newcommand{\C}{\mathbb{C}}
\newcommand{\N}{\mathbb{N}}

\newcommand{\T}{\mathbb{T}}

\newcommand{\U}{\mathrm{U}}

\newcommand{\cH}{\mathcal{H}}
\newcommand{\cB}{\mathcal{B}}
\newcommand{\cS}{\mathcal{S}}
\newcommand{\cA}{\mathcal{A}}
\newcommand{\cV}{\mathcal{V}}
\newcommand{\cT}{\mathcal{T}}

\newcommand{\mfA}{\mathfrak{A}}
\newcommand{\mfB}{\mathfrak{B}}
\newcommand{\mfC}{\mathfrak{C}}
\newcommand{\mfR}{\mathfrak{R}}

\newcommand{\End}{\mathrm{End}}

\newcommand{\NN}{{\mathrm{N}}}

\begin{document}

\title[Operator algebras and representations of Polish groups]{Operator algebras and unitary representations of Polish groups}

\author{Ra\'ul Quiroga-Barranco}
\address{Centro de Investigaci\'on en Matem\'aticas, Guanajuato, M\'exico}
\email{quiroga@cimat.mx}

\keywords{von Neumann algebras, $C^*$-algebras, Polish groups, unitary representations}

\subjclass[2020]{Primary 46L10; Secondary 46L35}

\maketitle

\begin{abstract}
	We study the relationship between operator algebras, $C^*$ and von Neumann, acting on a Hilbert space and unitary representations of topological groups on the same space. We obtain certain correspondences between both these families of objects. In particular, we prove that the study of Abelian subalgebras of a norm-separable $C^*$-algebra can be, in many aspects, reduced to the study of unitary representations of Polish groups. It is shown that our techniques have applications to the construction and classification of Abelian subalgebras of operator $C^*$-algebras on separable Hilbert spaces. 
\end{abstract}


\section{Introduction} 
\label{sec:intro}
We consider two of the most fundamental objects in functional analysis: operator algebras and unitary group representations. It is well known that these are strongly related through several constructions. We will focus in the following.
\begin{itemize}
	\item Every operator algebra $\mfA$ on a Hilbert space $\cH$, whether $C^*$ or von Neumann, yields the topological group $\U(\mfA)$ of its unitary elements which has a natural representation on $\cH$. 
	\item Every unitary representation of a topological group $G$ on a Hilbert space $\cH$ yields the von Neumann algebra $\End_G(\cH)$ of intertwining operators. And the latter usually contains many interesting $C^*$-subalgebras.
\end{itemize}
We intend to highlight the fact that the two previous constructions are, in some sense, dual or inverse of each other. With this general idea in mind, one of our goals is to show that every $C^*$-algebra can be written, up to some closure, as the von Neumann algebra of intertwining operators for a suitable unitary group representation. We intend to do so in a way useful for further applications.

First, we prove in Theorem~\ref{thm:C*algebras-group} that every $C^*$-subalgebra $\mfA$ of $\cB(\cH)$, where $\cH$ is a Hilbert space, is strong-operator dense in the von Neumann algebra of intertwining operators for a unitary representation on $\cH$ of some topological group. For such an algebra $\mfA$, we show that one can always take $\U(\mfA')$, the unitary group of the commutant of $\mfA$. This same result shows that such group is in fact the larger for which the density in the intertwining algebra can be achieved. Corollary~\ref{cor:AtoU(A')} thus introduces the assignment $\mfA \mapsto \U(\mfA')$ as an important tool with properties that establish the duality mentioned above.

These first results are in fact relatively easy to get from the Double Commutant Theorem. On the other hand, the whole unitary group of a von Neumann algebra, with its natural strong-operator topology, turns out to be quite large as a topological group. Hence, we proceed in the rest of the work to improve our choice of group as well as the conclusions we obtain from it. Our criteria is to consider objects that are accessible through direct integral decompositions. Of course, this requires to restrict ourselves to separable Hilbert spaces.

Norm-separable operator $C^*$-algebras are well known to be well suited to decompose with respect to direct integrals. With the help of this fact, we prove in Theorem~\ref{thm:C*algebra-intdecomp-group} that for a given operator $C^*$-algebra $\mfA$ we can find a subgroup $G$ of the unitary group $\U(\cH)$ whose intertwining algebra densely contains $\mfA$, and so that the $G$-action on $\cH$ decomposes with respect to the central direct integral decomposition of the von Neumann algebra generated by $\mfA$. We refer to section~\ref{sec:Decompositions} for the definitions involved, but we note here that such central direct integral decomposition is the most natural one for which $\mfA$ can be decomposed. In the rest of this introduction we will always refer to this direct integral decomposition. 

The group for which Theorem~\ref{thm:C*algebra-intdecomp-group} achieves its conclusions is in our case $G = \U(\mfB)_\NN$, the group of unitary operators, endowed with the norm topology, of a norm-separable $C^*$-algebra $\mfB$ that generates the commutant of $\mfA$. When this last condition holds we say that $\mfA$ and $\mfB$ are commutant dual (see Definition~\ref{def:commutant-dual}). Theorem~\ref{thm:C*algebra-intdecomp-group} greatly improves the results described above: the group $G = \U(\mfB)_\NN$ is much smaller, norm-separable and with a decomposable representation.

We then proceed to consider the realization of norm-separable operator $C^*$-algebras $\mfA$ as intertwining operators.  Theorem~\ref{thm:C*algebra-normseparable-intdecomp-group} shows that the norm-separability of $\mfA$ allows to decompose simultaneously the operators in $\mfA$ and $G$, where the latter is given as above. And this is achieved so that $\mfA$ consists of intertwining operators for $G$, both globally and on the components. This kind of result should be useful in the study of specific operator $C^*$-algebras.

Furthermore, Corollary~\ref{cor:normseparable-Abelian-intdecomp} proves that for $\mfA$ Abelian, its operators are diagonalizable and the decomposition for the group $G$ obtained in Theorem~\ref{thm:C*algebra-normseparable-intdecomp-group} has (almost everywhere) irreducible components. In particular, collecting some of the information given by Theorem~\ref{thm:C*algebra-normseparable-intdecomp-group} and Corollary~\ref{cor:normseparable-Abelian-intdecomp} we obtain the next result. We refer to the sections below for the needed definitions, here we observe that $\U(\cH)_\NN$ denotes the unitary group of $\cB(\cH)$ endowed with the norm topology.

\begin{theorem*}[Decomposition Theorem for Abelian $C^*$-Algebras]
	Let $\cH$ be a separable Hilbert space and let $\mfA \subset \cB(\cH)$ be an Abelian norm-separable $C^*$-subalgebra of $\cB(\cH)$. Then, there exist a Polish norm-closed subgroup $G \subset \U(\cH)_\NN$ such that with respect to the central direct integral decomposition of $\cH$ associated to $\overline{\mfA}$ and given by $(\cH_p)_{p \in X}$ over $(X,\mu)$, the following properties are satisfied.
	\begin{enumerate}
		\item $\mfA$ is strong-operator dense in the von Neumann algebra of diagonizable operators, and the latter is precisely $\End_G(\cH)$.
		\item There exists a family of unitary representations $(\rho_p)_{p \in X}$ of $G$ of the form $\rho_p : G \rightarrow \U(\cH_p)$, for almost every $p \in X$, for which we have
		\begin{enumerate}
			\item every $U \in G$ is decomposable with decomposition given by $U(p) = \rho_p(U)$, for almost every $p \in X$, and
			\item the unitary representation $\rho_p : G \rightarrow \U(\cH_p)$ is irreducible, for almost every $p \in X$.
		\end{enumerate}
	\end{enumerate}
\end{theorem*}

The previous result has an obvious importance of its own. However, it is also related to a problem that has appeared in the study of Toeplitz operators. In \cite{DOQJFA} it was constructed examples of Abelian $C^*$-algebras generated by Toeplitz operators acting on weighted Bergman spaces $\cA^2_\lambda(D)$ over bounded symmetric domains $D$. These examples were obtained from families of symbols invariant under suitable actions of subgroups $K$ of biholomorphisms of $D$. The subgroups $K$ considered in \cite{DOQJFA} have multiplicity-free unitary representations on the weighted Bergman spaces $\cA^2_\lambda(D)$, and their intertwining operators $\End_K(\cA^2_\lambda(D))$ contained a corresponding Abelian $C^*$-algebra $\mfA$ generated by Toeplitz operators. More precisely, the Abelian $C^*$-algebra $\mfA$ chosen was the intersection of the full Toeplitz algebra with $\End_K(\cA^2_\lambda(D))$. This sort of results has brought very active and interesting research in operator theory of analytic function spaces. We refer to \cite{DOQJFA,QVUnitBall1,QVUnitBall2}, and the references therein, for a sample of such research. Within this setup and notation, the following questions have been raised.

\begin{itemize}
	\item For $\mfA$ an Abelian $C^*$-algebra as above, is $\mfA$ strong-operator dense in the von Neumann algebra $\End_K(\cA^2_\lambda(D))$?
	\item Given an Abelian $C^*$-subalgebra $\mfA$ of the full Toeplitz algebra acting on $\cA^2_\lambda(D)$, is there a biholomorphism subgroup $K$ acting on $D$ such that $\mfA$ is strong-operator dense in $\End_K(\cA^2_\lambda(D))$?
	\item If this can be achieved, is it possible to decompose both $\mfA$ and the representation of $K$ with respect to some direct integral decomposition of $\cH$ so as to diagonalize the former?
\end{itemize}

The first question was answered affirmatively in \cite{DOQJFA} for the case of $K$ a compact subgroup of the biholomorphism group of $D$. The rest of the cases and other questions remained open except for some particular examples. With this respect, the general discussion considered in this work is clearly an extension of the previous problems from analytic function spaces to arbitrary separable Hilbert spaces. Furthermore, our Decomposition Theorem for Abelian $C^*$-Algebras stated above gives affirmative answers to the questions just formulated even for the more general setup of operators on an arbitrary separable Hilbert space. The only restriction is that one has to consider Polish norm-closed subgroups of the full unitary operator group instead of unitary representations induced from biholomorphism subgroups. Nevertheless, the main properties expected are achieved: for example, the decomposition of the unitary representation into irreducible components. This is not a trivial conclusion since the group $G$ above is in general far from being a Lie group or even a locally compact group.

We further develop our theory by providing in Theorem~\ref{thm:maximal-vN-group} a characterization of norm-separable Abelian $C^*$-algebras that generate maximal von Neumann algebras in terms of corresponding Polish norm-closed groups of unitary elements. This allows us to formulate in Corollary~\ref{cor:maximal-vN-group} a correspondence, bijective up to some strong-operator closure, between norm-separable $C^*$-algebras and Polish norm-closed subgroups when both are maximal Abelian in a suitable sense. This is again motivated by the work on weighted Bergman spaces and Abelian $C^*$-algebras generated by Toeplitz operators. We refer to \cite{QVUnitBall1,QVUnitBall2} where it was proved that for the unit ball and the corresponding Siegel domain the maximal Abelian subgroups of biholomorphisms yield Abelian $C^*$-algebras generated by Toeplitz operators.

Next, we consider the problem of describing the Abelian $C^*$-subalgebras of a given norm-separable operator $C^*$-algebra $\cT$ acting on a separable Hilbert space $\cH$. As noted in the previous discussion, such Abelian algebras can be obtained from Polish norm-closed subgroups of $\U(\cH)_\NN$. We show in Corollary~\ref{cor:Abelian-subalgebras-cT} and Theorem~\ref{thm:Abelian-subalg-vs-multiplicity-free} that the family of groups that have to be considered may be limited by the properties of the ambient algebra $\cT$ that contains the Abelian algebras to be studied. 

On the other hand, and as a by-product of our theory, we obtain some results for unitary representations of Polish norm-closed subgroups $G$ of the unitary group $\U(\cH)_\NN$ of a separable Hilbert space $\cH$. In the first place, Corollary~\ref{cor:group-decomposition} shows that the representation on $\cH$ of any such group $G$ decomposes with respect to the central direct integral decomposition of $\cH$ associated to the von Neumann algebra generated by $G$. In the second place, Corollary~\ref{cor:group-decomposition-multfree} proves that if the representation of such $G$ on $\cH$ is multiplicity-free, then the decomposition of $G$ so obtained is given by irreducible pieces. These results have an obvious interest for the theory of topological groups. The reason is that one can usually achieve this sort of conclusions for Lie or locally compact groups of a suitable type. However, these corollaries consider arbitrary norm-separable norm-closed subgroups of $\U(\cH)_\NN$.

In order to better understand the difference between our approach with Polish groups and the usual techniques from Lie theory, we present in the last subsection of this work a comparison between the use of compact groups and the corresponding Polish groups that replace them. This also allows to see that, even in the simplest of examples, the Polish groups that we study are not locally compact.

As for the content of this work, sections~\ref{sec:U(H)-Polish-algebras} and \ref{sec:Decompositions} provide the background and notation coming from well known facts of operator theory. Section~\ref{sec:operator-alg-topgroups} contains our preliminary results that are obtained without the separability assumption on Hilbert spaces, while section~\ref{sec:operator-alg-Polishgroups} presents the main results discussed in this introduction.

\section{Operator algebras and unitary groups}
\label{sec:U(H)-Polish-algebras}

\subsection{Some facts and conventions}
A few elementary remarks are in order for this work. We recall that the Double Commutant Theorem states that, for any Hilbert space $\cH$, if $\mfA$ is a $*$-subalgebra of $\cB(\cH)$ that contains the identity operator, then $\mfA''$ is the von Neumann algebra generated by $\mfA$, and that it is also its closure with respect to both the strong-operator and weak-operator topologies. For this reason, we will also denote by $\overline{\mfA} = \mfA''$ such common closure. We also recall that for any subset $\cS \subset \cB(\cH)$ we have that $\cS' = \cS'''$ and $\cS'$ is a von Neumann algebra if $\cS$ is self-adjoint. In particular, if $\mfA$ is a $C^*$-subalgebra of $\cB(\cH)$, then $\mfA'$ is a von Neumann algebra. When we consider a $C^*$-subalgebra of $\cB(\cH)$ we will assume that it contains the identity operator. Finally, we recall the basic fact that every element in a $C^*$-algebra is the linear combination of (four) unitary elements in the algebra. We will use all of these fundamental facts and assumptions in the rest of this work without further mention. 

\subsection{The unitary group of an operator algebra}
From now on, we will denote by $\U(\cH)$ the group of unitary elements in the von Neumann algebra $\cB(\cH)$ of bounded operators acting on a Hilbert space $\cH$. We recall that the weak-operator, strong-operator and strong-operator-$*$ topologies restricted to $\U(\cH)$ are all the same. That the first two topologies coincide on $\U(\cH)$ follows, for example, from \cite[Exercise~5.7.5]{KRvolI}. Then, the strong-operator continuity of $A \mapsto A^*$ on the set of normal operators (see~\cite[Remark~2.5.10]{KRvolI}) implies that all three topologies coincide on $\U(\cH)$. From now on, we will consider $\U(\cH)$ endowed with this topology.

On the other hand, it is well known (see \cite[Remark~2.5.10]{KRvolI}) that the product $(A,B) \mapsto AB$ in $\cB(\cH)$ is strong-operator continuous when the first variable lies in a bounded set, and so it is continuous on $\U(\cH)$. Since the map $A \mapsto A^*$ is strong-operator continuous on $\U(\cH)$, it follows that $\U(\cH)$ is a topological group.

Recall that a Polish space is a separable completely metrizable topological space. With this notation, if $\cH$ is separable, then the closed unit ball $(\cB(\cH))_1$ of $\cB(\cH)$ with the strong-operator-$*$ topology is a Polish space and $\U(\cH)$ is closed in it (see~\cite[Exercise~14.4.9]{KRvolII}). It follows that, for $\cH$ a separable Hilbert space, $\U(\cH)$ is a Polish group: i.e.~a topological group whose topology is Polish.

The group $\U(\cH)$ is weak-operator dense in the closed unit ball $(\cB(\cH))_1$ when $\cH$ is infinite-dimensional (see \cite[Exercise~5.7.12(i)]{KRvolI}), in which case $\U(\cH)$ is a non-compact topological group. This is a well known fact that is in clear contrast with the behavior in finite-dimensional cases. 

Let us now consider a von Neumann algebra $\mfR$ acting on a Hilbert space $\cH$. We will denote by $\U(\mfR) = \U(\cH) \cap \mfR$ the group of unitary elements in $\mfR$, which we will endow in the rest of this work with the topology induced from $\U(\cH)$. In particular, $\U(\mfR)$ is a closed topological subgroup of $\U(\cH)$. Furthermore, we have the next result that follows from the previous discussion, \cite[Exercise~14.4.9]{KRvolII} and the (easy to prove) fact that the closed unit ball in every von Neumann algebra is strong-operator-$*$ closed.

\begin{proposition}\label{prop:U(R)-topgroup-vN}
	Let $\mfR$ be a von Neumann algebra acting on a Hilbert space $\cH$. Then, the group $\U(\mfR)$ is a topological group which is Polish when $\cH$ is separable. Furthermore, $\U(\mfR)$ is a closed subset of both $\U(\cH)$ and $(\mfR)_1$, where the latter is the closed unit ball in $\mfR$ and it is given the strong-operator-$*$ topology. In particular, $\U(\mfR)$ is strong-operator-$*$ closed in $\cB(\cH)$. Also, the weak-operator, strong-operator and strong-operator-$*$ topologies all coincide on~$\U(\mfR)$.
\end{proposition}

For any Hilbert space $\cH$, it is well known that the group of invertible elements in $\cB(\cH)$ is a topological group with respect to the norm topology, with the subset of unitary operators being a closed topological subgroup. In particular, if $\mfA$ is a $C^*$-subalgebra of $\cB(\cH)$, then the group of unitary elements in $\mfA$ is a topological group which is norm closed in $\mfA$. We will denote by $\U(\mfA)_\NN$ the group of unitary elements in the $C^*$-algebra $\mfA$ endowed with the norm topology, where the subindex $\NN$ is placed to emphasize the use of such topology. This will avoid confusion in the case where the $C^*$-algebra is a von Neumann algebra as well. As a matter of fact, for every von Neumann algebra $\mfR$ acting on the Hilbert space $\cH$ both $\U(\mfR)$ and $\U(\mfR)_\NN$ are topological groups with the same underlying set: the unitary elements of $\mfR$. However, the former carries the strong-operator topology while the latter carries the norm topology. 

Some of the previous remarks about $\U(\mfA)_\NN$ are collected in the next result.

\begin{proposition}\label{prop:U(A)-topgroup-C*}
	Let $\cH$ be a Hilbert space and $\mfA$ be a $C^*$-subalgebra of $\cB(\cH)$. Then, the group $\U(\mfA)_\NN$ is a topological group which is Polish when $\mfA$ is norm-separable. Furthermore, $\U(\mfA)_\NN$ is norm closed in $\cB(\cH)$.
\end{proposition}

We recall the following consequence of Kaplansky Density Theorem (see \cite[Corollary~5.3.7]{KRvolI}).

\begin{proposition}\label{prop:KaplanskyUnitaryDensity}
	Let $\cH$ be a Hilbert space and $\mfA$ be a $C^*$-subalgebra of $\cB(\cH)$. Then, the group $\U(\mfA)_\NN$ is strong-operator dense in the group $\U(\overline{\mfA})$. In other words, we have $\overline{\U(\mfA)_\NN} = \U(\overline{\mfA})$, where the first closure is taken in $\U(\overline{\mfA})$ with respect to its (strong-operator) topology
\end{proposition}

\section{Operator $C^*$-algebras and topological groups}
\label{sec:operator-alg-topgroups}
\subsection{Intertwining maps of topological groups}
A unitary representation of a topological group $G$ on a Hilbert space $\cH$ is a continuous action $\rho : G \times \cH \rightarrow \cH$ such that, for every fixed $g \in G$, the map $\rho(g) : \cH \rightarrow \cH$ given by $\rho(g)(x) = \rho(g,x)$ is unitary. The continuity of the action is well known to be equivalent to the continuity of the homomorphism $\rho : G \rightarrow \U(\cH)$ given by $g \mapsto \rho(g)$, where $\U(\cH)$ is endowed with the strong-operator topology, the one that we have agreed to use on $\U(\cH)$. For this reason, we will use $\rho$ to denote both the action map and the homomorphism, just considered, interchangeably without further mention.

We will consider two fundamental examples for our purposes.
\begin{enumerate}
	\item Let $\mfR$ be a von Neumann algebra acting on the Hilbert space $\cH$, then the action $\U(\mfR) \times \cH \rightarrow \cH$ yields a unitary representation of the topological group $\U(\mfR)$ on $\cH$.
	\item Let $\mfA$ be a $C^*$-subalgebra of $\cB(\cH)$, where $\cH$ is a Hilbert space, then the action of $\U(\mfA)_\NN \times \cH \rightarrow \cH$ yiels a unitary representation of the topological group $\U(\mfA)_\NN$ on $\cH$.
\end{enumerate}

For a given unitary representation $\rho$ of a topological group $G$ on a Hilbert space $\cH$, a bounded operator $T \in \cB(\cH)$ that satisfies $T \circ \rho(g) = \rho(g) \circ T$, for every $g \in G$, is called an intertwining operator. The space of all such intertwining operators will be denoted by $\End_G(\cH)$. Note that this space is precisely the commutant $\rho(G)'$ of $\rho(G)$ in $\cB(\cH)$, and so $\End_G(\cH)$ is a von Neumann algebra acting on $\cH$. 

\subsection{Operator $C^*$-algebras as intertwining maps (topological groups)}
The next result proves that every $C^*$-algebra of operators acting on a Hilbert space is strong-operator dense in the von Neumann algebra of intertwining operators for some unitary representation. Moreover, this result exhibits the largest topological group for which this holds. 

\begin{theorem}\label{thm:C*algebras-group}
	Let $\cH$ be a Hilbert space and $\mfA$ be a $C^*$-subalgebra of $\cB(\cH)$. Then, the following properties hold.
	\begin{enumerate}
		\item The $C^*$-algebra $\mfA$ is strong-operator dense in $\End_{\U(\mfA')}(\cH)$: i.e.~we have $\overline{\mfA} = \End_{\U(\mfA')}(\cH)$.
		\item If $\rho : G \rightarrow \U(\cH)$ is a unitary representation of a topological group $G$, then we have
		\begin{enumerate}
			\item $\mfA \subset \End_G(\cH)$ if and only if the von Neumann algebra generated by $\rho(G)$ is contained in $\mfA'$ if and only if $\rho(G) \subset \U(\mfA')$.
			\item $\overline{\mfA} = \End_G(\cH)$ if and only if the von Neumann algebra generated by $\rho(G)$ is $\mfA'$. This is the case when $\rho(G)$ is weak-operator dense in~$\U(\mfA')$.
		\end{enumerate}
	\end{enumerate}
	Furthermore, if $\cH$ is separable, then $\U(\mfA')$ is a Polish group.
\end{theorem}
\begin{proof}
	Since $\mfA'$ is linearly generated by $\U(\mfA')$ it follows that $\overline{\mfA} = \mfA'' = \U(\mfA')' = \End_{\U(\mfA')}(\cH)$. This proves the first claim.
	
	For the second claim, let us now consider a unitary representation $\rho : G \rightarrow \U(\cH)$. Then, we have the following equivalences
	\[
		\mfA \subset \End_G(\cH) 
				\Longleftrightarrow \mfA \subset \rho(G)'
				\Longleftrightarrow \rho(G)'' \subset \mfA',
	\]
	which prove the first equivalence in (2a). If $\rho(G)'' \subset \mfA'$ holds, then we have $\rho(G) \subset \U(\mfA')$ because $\rho(G)$ is a subset of both $\rho(G)''$ and $\U(\cH)$. Conversely, the inclusion $\rho(G) \subset \U(\mfA')$ implies $\rho(G)'' \subset \U(\mfA')'' = \mfA'$. This completes the proof of~(2a). 
	
	The equivalence in (2b) is obtained once we note that the above displayed equivalences hold after replacing $\mfA$ with $\overline{\mfA}$ and the inclusions with identities. 
	
	Finally, let us assume that $\rho(G)$ is weak-operator dense in $\U(\mfA')$. Hence, if we denote by $\cV$ the space of finite linear combinations of elements in $\rho(G)$, then it follows that the weak-operator closure $\overline{\cV}$ is precisely $\mfA'$. Next, we note that $\cV$ is a $*$-algebra since $\rho(G)$ is a group consisting of unitary elements. Hence, we conclude that
	\[
		\rho(G)'' = \cV'' = \overline{\cV} = \mfA',
	\]
	i.e.~the von Neumann algebra generated by $\rho(G)$ is $\mfA'$. This completes the proof of~(2b). The last claim follows from Proposition~\ref{prop:U(R)-topgroup-vN}.
\end{proof}

\begin{remark}\label{rmk:C*algebras-group-counterexample}
	The converse in the second claim of Theorem~\ref{thm:C*algebras-group}(2b) does not hold in general as the next easy example shows. For any Hilbert space $\cH$, let us take $\mfA = \cB(\cH)$ and $G = \{I\}$ with its natural action on $\cH$, so that we obtain $\U(\mfA') = \T I$. For this setup, we have $\mfA = \End_G(\cH)$ but $G$ is a proper closed subgroup of $\U(\mfA')$.
\end{remark}

\begin{remark}\label{rmk:C*algebras-group}
	For a given $C^*$-algebra $\mfA$ of operators acting on a Hilbert space $\cH$, Theorem~\ref{thm:C*algebras-group} proves that there is always a solution to the problem of finding a unitary representation $\rho$ of a topological group $G$ such that $\overline{\mfA} = \End_G(\cH)$, and that such solution can be assumed to be given by a Polish group when $\cH$ is separable. Furthermore, it also exhibits conditions on such unitary representations that can be considered as upper and lower bounds on $\rho$ in terms of $\mfA$. Theorem~\ref{thm:C*algebras-group}(2a) imposes an upper bound condition, since it shows that any solution $\rho$ must be such that $\rho(G) \subset \U(\mfA')$, i.e.~the unitary operators given by $\rho$ must lie among the unitary elements of the von Neumann algebra $\mfA'$. Theorem~\ref{thm:C*algebras-group}(2b) implies a lower bound condition, since it shows that any solution $\rho$ must be such that $\rho(G)$ is large enough to generate the von Neumann algebra $\mfA'$. At the same time, it gives a condition under which this occurs: the weak operator density of $\rho(G)$ in the group of unitary elements in $\mfA'$.
\end{remark}

From Remark~\ref{rmk:C*algebras-group} we conclude that the topological group $\U(\mfA')$ considered in Theorem~\ref{thm:C*algebras-group} is naturally associated to any $C^*$-subalgebra $\mfA$ of $\cB(\cH)$. This is so from the viewpoint of describing the elements of $\mfA$ as intertwining operators. In the next result we state such association as an assignment to stress its relevance.

\begin{corollary}\label{cor:AtoU(A')}
	For a Hilbert space $\cH$, the assignment given by
	\[
		\mfA \longmapsto \U(\mfA'),
	\]
	maps the class of $C^*$-subalgebras of $\cB(\cH)$ to the class of closed topological subgroups of $\U(\cH)$ and satisfies the following properties.
	\begin{enumerate}
		\item For every $C^*$-subalgebra $\mfA \subset \cB(\cH)$, the topological group $\U(\mfA')$ is the largest subgroup of $\U(\cH)$ such that $\End_{\U(\mfA')}(\cH) = \overline{\mfA}$.
		\item The assignment restricted to the subclass of von Neumann algebras acting on $\cH$ is injective. More generally, $\U(\mfA_1') = \U(\mfA_2')$ if and only if $\mfA_1$ and $\mfA_2$ generate the same von Neumann algebra.
	\end{enumerate}
\end{corollary}
\begin{proof}
	The first claim is an immediate consequence of Theorem~\ref{thm:C*algebras-group}. On the other hand, for a given pair of $C^*$-subalgebras $\mfA_1$ and $\mfA_2$ of $\cB(\cH)$ we have $\U(\mfA_1') = \U(\mfA_2')$ if and only if $\mfA_1' = \mfA_2'$, which is equivalent to $\overline{\mfA}_1 = \overline{\mfA}_2$. This proves the second claim.
\end{proof}

\begin{remark}\label{cor:AtoU(A')-counterexamples}
	It is clear that the assignment from Corollary~\ref{cor:AtoU(A')} is not injective in general since, for $\cH$ infinite-dimensional, there are plenty of $C^*$-algebras with the same weak-operator closure. On the other hand, if $\mfR \subset \cB(\cH)$ is a von Neumann algebra, then we have
	\[
		\U(\mfR') \supset \U(\cB(\cH)') = \T I.
	\]
	Hence, for any discrete subgroup $G \subset \T$ the group $G I$ is not of the form $\U(\mfR')$ for any von Neumann algebra $\mfR$. It follows that the assignment is not surjective either.
\end{remark}
 
\begin{remark}\label{rmk:subgroups-of-U(A')}
	One of the main problems that we consider in this work is the study of the relationship between $C^*$-algebras of $\cB(\cH)$ and unitary representations on the Hilbert space $\cH$.	Theorem~\ref{thm:C*algebras-group} and Corollary~\ref{cor:AtoU(A')} show that, for a given $C^*$-algebra $\mfA$, the most important group to consider is $\U(\mfA')$, since this is where any other group from which $\mfA$ can be built is found. In the rest of this work we will make use of this viewpoint together with the theory of operator algebra decompositions with respect to direct integrals.
\end{remark}

\section{Decompositions of spaces and algebras}
\label{sec:Decompositions}
\subsection{Direct integrals of Hilbert spaces}
We will use the notion of direct integrals of Hilbert spaces as defined in \cite{KRvolII} (see also \cite{DixmierCAlgebras,DixmierVonNeumann,TakesakiI}), which we recall in the following definition. As it is well known, the measure theoretic techniques involved requires to consider only separable Hilbert spaces. We will be careful to specify this condition when needed. On the other hand, such condition will allow us to improve Theorem~\ref{thm:C*algebras-group} so that Polish groups can be obtained.

\begin{definition}\label{def:direct-integral}
	Let $\cH$ be a separable Hilbert space. Assume that we are given a locally compact Polish space $X$ with a Radon measure $\mu$ on it, and a family $(\cH_p)_{p \in X}$ of separable Hilbert spaces. We say that $\cH$ is the direct integral of $(\cH_p)_{p \in X}$ over $(X,\mu)$ and write
	\begin{equation} \label{eq:direct-integral}
		\cH = \int^\oplus_X \cH_p \dif \mu(p),
	\end{equation}
	if there is a map that assigns to each $x \in \cH$ a function $X \longrightarrow \bigsqcup_{p \in X} \cH_p$ (which will be denoted with the same symbol) that satisfies $x(p) \in \cH_p$, for all $p \in X$, so that the following properties hold.
	\begin{enumerate}
		\item For every $x, y \in \cH$ the function $p \mapsto \langle x(p), y(p) \rangle$ belongs to $L^1(X,\mu)$ and it satisfies
		\[
		\langle x, y \rangle 
		= \int_X \langle x(p), y(p) \rangle \dif \mu(p).
		\]
		\item If for some choice of elements $(x_p)_{p \in X}$ with $x_p \in \cH_p$, for all $p \in X$, we have that $p \mapsto  \langle x_p , y(p) \rangle$ belongs to $L^1(X,\mu)$ for every $y \in \cH$, then there exists $x \in \cH$ such that $x_p = x(p)$ for almost every $p \in X$.
	\end{enumerate}
\end{definition}

\begin{remark}\label{rmk:direct-integral}
	We will follow the aforementioned references for the well known properties of direct integrals of Hilbert spaces. For simplicity, in the rest of this work when we speak of a direct integral of $(\cH_p)_{p \in X}$ over $(X,\mu)$ we will assume that the properties stated in Definition~\ref{def:direct-integral} hold. The same convention will apply when we write equation~\eqref{eq:direct-integral}. In particular, we will always assume that any direct integral is taken over a locally compact Polish space and that its measure is Radon. The existence theorems that we will use in this work ensure that this is possible (see \cite{DixmierVonNeumann,KRvolII}).
\end{remark}

Let us consider a separable Hilbert $\cH$ space which is the direct integral of $(\cH_p)_{p \in X}$ over $(X,\mu)$. It is well known that this yields two important types of operators. Let us consider $T \in \cB(\cH)$, a bounded operator on $\cH$. We will say that $T$ is \textit{decomposable} if there is a family $(T(p))_{p \in X}$ with $T(p) \in \cB(\cH_p)$, for every $p \in X$, such that for every $x \in \cH$ we have $(Tx)(p) = T(p)x(p)$, for almost every $p \in X$. In this case, we will write
\[
T = \int^\oplus_X T(p) \dif \mu(p),
\]
and we will also say that the family $(T(p))_{p \in X}$ of operators or the assignment $p \mapsto T(p)$, where $p \in X$, is the decomposition of $T$. Such operator $T$ is called \textit{diagonalizable} if we further have $T(p) = f(p) I_p$ for $I_p$ the identity operator on $\cH_p$, for almost every $p \in X$, where $f : X \rightarrow \C$ is a function. It is well known (see \cite[Section~14.1]{KRvolII}) that in the latter case $f$ is measurable essentially bounded. Conversely, every function $f \in L^\infty(X,\mu)$ defines a diagonalizable operator that is denoted by $M_f$. 

Let $\cH$ be a separable Hilbert space which is the direct integral of $(\cH_p)_{p \in X}$ over $(X,\mu)$. If $\mfR$ is the set of decomposable operators for this direct integral, then $\mfR$ is a von Neumann algebra acting on $\cH$. Furthermore, its commutant $\mfR'$ is precisely the space of diagonalizable operators and is thus a von Neumann algebra as well. We refer to \cite[Section~14.1]{KRvolII} for the corresponding proofs. It is a fundamental result that these facts have a sort of converse: any Abelian von Neumann algebra yields a direct integral decomposition whose diagonalizable operators are precisely those given by the Abelian algebra. The next result is an extension of this last fact and it follows easily from \cite[Theorem~14.2.1]{KRvolII} (see also \cite{DixmierVonNeumann}).

\begin{proposition}\label{prop:Abelian-direct-integral}
	Let $\cH$ be a separable Hilbert space and $\mfA$ be an Abelian $C^*$-subalgebra of $\cB(\cH)$. Then, there exist a measure space $(X,\mu)$, where $X$ is a locally compact Polish space and $\mu$ is a Radon measure in $X$, and a family $(\cH_p)_{p \in X}$ of Hilbert spaces such that the following properties are satisfied.
	\begin{enumerate}
		\item $\cH$ is the direct integral of $(\cH_p)_{p \in X}$ over $(X,\mu)$.
		\item $\overline{\mfA}$ and $\mfA'$ are the von Neumann algebras of diagonalizable and decomposable operators, respectively.
	\end{enumerate}
\end{proposition}

\subsection{Decompositions of operator algebras}
It is a fundamental fact that the description of a separable Hilbert space $\cH$ as a direct integral induces a corresponding decomposition on suitable operator algebras acting on $\cH$. This holds for both $C^*$-subalgebras and von Neumann subalgebras of $\cB(\cH)$ when either consist of decomposable operators or (equivalently) when its elements commute with the diagonalizable operators. We make this remark precise by recalling some facts from~\cite[Section~14.1]{KRvolII}. Further details and proofs can be found in this reference.

\begin{definition}\label{def:C*-decomposable}
	Let $\cH$ be a separable Hilbert space which is the direct integral of $(\cH_p)_{p \in X}$ over $(X,\mu)$ and $\mfA$ be a $C^*$-subalgebra of the von Neumann algebra of decomposable operators. We will say that $\mfA$ is a decomposable $C^*$-algebra if there exist a family $(\varphi_p)_{p \in X}$ of $C^*$-algebra representations $\varphi_p : \mfA \rightarrow \cB(\cH_p)$, where $p \in X$, such that for every $A \in \mfA$ we have $A(p) = \varphi_p(A)$, for almost every $p \in X$. In this case, the family $(\varphi_p)_{p \in X}$ is called a decomposition of the $C^*$-algebra $\mfA$.
\end{definition}

By its nature, a decomposition of a $C^*$-algebra may be considered only over a conull subset. On the other hand, these always exist in the norm-separable case as we now state (see~\cite[Theorem~14.1.13]{KRvolII}).

\begin{proposition}\label{prop:C*normseparable-decomposable}
	If $\cH$ is a separable Hilbert space which is the direct integral of $(\cH_p)_{p \in X}$ over $(X,\mu)$ and $\mfA$ is a norm-separable $C^*$-subalgebra of the von Neumann algebra of decomposable operators, then $\mfA$ is a decomposable $C^*$-algebra.
\end{proposition}

\begin{remark}\label{rmk:decompositions-unit}
	Following \cite{KRvolI,KRvolII} we will assume in this work that the homomorphisms, and so the representations, of $C^*$-algebras map the unit onto the unit. With this restriction in place, Definition~\ref{def:C*-decomposable} ought to be stated by requiring $\varphi_p$ to exist only for $p \in X \setminus N$, where $N$ is some null subset of $X$. On the other hand, the conditions that these representations satisfy have to hold almost everywhere. In particular, a decomposition of a $C^*$-algebra continues to be so even after replacing or removing its elements on a null subset. Hence, whenever Definition~\ref{def:C*-decomposable} is applied we will assume that the representations $\varphi_p$ are only given in some conull subset of $X$, even though they are indexed for all $p \in X$. This convention is introduced to simplify our statements. Nevertheless, we will explicitly indicate its use whenever it is important to avoid confusion. In particular, Proposition~\ref{prop:C*normseparable-decomposable} will be assumed to hold under such convention.
\end{remark}

We now consider decompositions of von Neumann algebras. As it is well known, it turns out that such decompositions are dependent upon corresponding ones for $C^*$-algebras. 

\begin{definition}\label{def:vonNeumann-decomposition}
	Let $\cH$ be a separable Hilbert space which is the direct integral of $(\cH_p)_{p \in X}$ over $(X,\mu)$. A von Neumann algebra $\mfR$ acting on $\cH$ is called decomposable, with decomposition $(\mfR_p)_{p \in X}$, if there exist a strong-operator dense norm-separable $C^*$-subalgebra $\mfA \subset \mfR$  for which there is a decomposition $(\varphi_p)_{p \in X}$ of $\mfA$ such that the $C^*$-algebra $\varphi_p(\mfA)$ is strong-operator dense in $\mfR_p$, for almost every $p \in X$. 
\end{definition}

As in the case discussed in Remark~\ref{rmk:decompositions-unit}, decompositions of von Neumann algebras may be specified almost everywhere. It is well known (see \cite[Theorem~14.1.16]{KRvolII}) that they always exist when the elements of the algebra are decomposable, as we now state.

\begin{proposition}\label{prop:vonNeumann-decomposition}
	Let $\cH$ be a separable Hilbert space which is the direct integral of $(\cH_p)_{p \in X}$ over $(X,\mu)$ and $\mfR$ be a von Neumann subalgebra of the algebra of decomposable operators. Then, $\mfR$ is decomposable and admits an almost everywhere unique decomposition $(\mfR_p)_{p \in X}$.
\end{proposition}

\begin{remark}\label{rmk:vonNeumann-decomposition}
	We explain some of the details relevant to the previous definition and proposition. In the first place, it is well known (see~\cite[Lemma~14.1.17]{KRvolII}) that every von Neumann algebra acting on a separable Hilbert space contains a strong-operator dense norm-separable $C^*$-subalgebra. Next, as we stated in Proposition~\ref{prop:C*normseparable-decomposable}, every norm-separable $C^*$-algebra is decomposable. Finally, the strong-operator closure of $\varphi_p(\mfA)$ is independent of the choices involved for almost every $p \in X$ (see~\cite[Lemma~14.1.15]{KRvolII}). 
\end{remark}

With the notation of Definition~\ref{def:vonNeumann-decomposition}, if $\mfC$ is the center of $\mfR$, then we trivially have $\mfR \subset \mfC'$ and so $\mfR$ consists of decomposable operators for the direct integral decomposition of $\cH$ given by the Abelian von Neumann algebra $\mfC$. Hence, Propositions~\ref{prop:Abelian-direct-integral} and \ref{prop:vonNeumann-decomposition} allow us to formulate the following definition.

\begin{definition}\label{def:vonNeumann-centraldecomposition}
	Let $\cH$ be a separable Hilbert space and $\mfR$ be a von Neumann algebra acting on $\cH$ with center $\mfC$. The central direct integral decomposition of $\cH$ associated to $\mfR$ and the central decomposition of $\mfR$ are those given by the Abelian von Neumann algebra $\mfC$, respectively.
\end{definition}

\begin{remark}\label{rmk:support-Radon}
	Let $X$ be a locally compact Polish space carrying a Radon measure $\mu$ and let us denote by $X_0$ the support of $\mu$. In other words, $X \setminus X_0$ is the largest null open subset of $X$. In particular, $X_0$ is closed in $X$ and so it is a locally compact Polish space as well. We also conclude that for every measurable subset $E$ of $X$ we have $\mu(E) = \mu(E \cap X_0)$. It follows that the restriction of $\mu$ to $X_0$ is a Radon measure which is positive on non-empty open subsets of $X_0$. It is clear that all the measure theory spaces attached to $X$ and $X_0$ are naturally isomorphic by the restriction map $f \mapsto f|_{X_0}$. In other words, for measure theoretic purposes, we can replace any locally compact Polish space carrying a Radon measure with a closed subspace to assume that $\mu$ is positive on non-empty open subsets.
	
	We wish to take this elementary fact one step further in a way that will be useful latter on. More precisely, we observe that all the constructions and definitions considered in this section are given by conditions that need to be satisfied almost everywhere. In particular, these are not affected by removing null sets. Hence, we can assume in all the definitions and results stated so far that the Radon measure $\mu$ on the locally compact Polish space $X$ is positive on non-empty subsets. This is achieved as in the previous paragraph by replacing $X$ with the support $X_0$ of $\mu$ which only discards the null open set $X \setminus X_0$. We will follow this assumption without further mention in the rest of this work while being careful to make it explicit whenever any confusion may arise.
\end{remark}

We state the following simple result whose proof we provide for the sake of completeness. It yields an example where the assumption established in Remark~\ref{rmk:support-Radon} is useful.

\begin{lemma}\label{lem:Cb-into-BL2}
	Let $X$ be a locally compact Polish space carrying a Radon measure $\mu$ which is positive on non-empty open subsets of $X$. Then, the natural mappings $C_b(X) \rightarrow L^\infty(X,\mu) \rightarrow \cB(L^2(X,\mu))$ given by $f \mapsto f \mapsto M_f$ are $C^*$-algebra isomorphisms onto their corresponding images.
\end{lemma}
\begin{proof}
	The given map $L^\infty(X,\mu) \rightarrow \cB(L^2(X,\mu))$ is well known to be an isomorphism of $C^*$-algebras onto its image. This fact does not require the positivity condition on~$\mu$.
	
	The assignment $C_b(X) \rightarrow L^\infty(X,\mu)$ is considered as mapping a bounded continuous function $f$ onto its function class modulo null sets. Note that this map is clearly a homomorphism of $C^*$-algebras. Hence, it is enough to prove its injectivity. Let $f \in C_b(X)$ be such that its class in $L^\infty(X,\mu)$ is zero. Hence, $X \setminus f^{-1}(0)$ is a null open set and by the assumption on $\mu$ it is empty. This implies that $f^{-1}(0) = X$ and so $f$ is the function zero as an element of $C_b(X)$.
\end{proof}

\section{Operator $C^*$-algebras and Polish groups}
\label{sec:operator-alg-Polishgroups}
\subsection{Operator $C^*$-algebras as intertwining operators (Polish groups)}
Let $\cH$ be a Hilbert space and $\mfB$ be a $C^*$-subalgebra of $\cB(\cH)$. As noted before, the topological group $\U(\mfB)_\NN$ has a natural action on $\cH$ given by the inclusion homomorphism $\U(\mfB)_\NN \hookrightarrow \U(\cH)$, which is clearly continuous since the norm topology is finer than the strong-operator topology. Hence, we obtain a unitary representation of $\U(\mfB)_\NN$ on $\cH$. Furthermore, if $\mfB$ is norm-separable, then the topological group $\U(\mfB)_\NN$ is Polish. By using the unitary representation of the group $\U(\mfB)_\NN$ on $\cH$ we improve, in the next result, the conclusions from Theorem~\ref{thm:C*algebras-group} in a way that will be useful for direct integral decompositions. We recall from Remark~\ref{rmk:vonNeumann-decomposition} that a $C^*$-subalgebra $\mfB$ as considered below always exists.

\begin{proposition}\label{prop:C*algebra-group-normtopo}
	Let $\cH$ be a separable Hilbert space and $\mfA$ be a $C^*$-subalgebra of $\cB(\cH)$. If $\mfB$ is a strong-operator dense norm-separable $C^*$-subalgebra of $\mfA'$, then the following properties hold.
	\begin{enumerate}
		\item The $C^*$-algebra $\mfA$ is strong-operator dense in $\End_{\U(\mfB)_\NN}(\cH)$: i.e.~we have $\overline{\mfA} = \End_{\U(\mfB)_\NN}(\cH)$.
		\item The Polish group $\U(\mfB)_\NN$ is strong-operator dense in $\U(\mfA')$: i.e.~we have $\overline{\U(\mfB)_\NN} = \U(\mfA')$, where the closure is taken in $\U(\mfA')$ with respect to its (strong-operator) topology.
	\end{enumerate}
	Furthermore, the von Neumann algebras generated by $\mfA$ and $\mfB$ are the commutant of each other. More precisely, we have $\overline{\mfB} = \mfA'$ and $\overline{\mfA} = \mfB'$. In particular, $\overline{\mfA} \cap \overline{\mfB}$ is the center of both of these von Neumann algebras.
\end{proposition}
\begin{proof}
	Let $\mfB$ be a strong-operator dense $C^*$-subalgebra of the von Neumann algebra $\mfA'$. Hence, (2) follows from Proposition~\ref{prop:KaplanskyUnitaryDensity}. Such strong-operator density clearly implies the first identity in
	\[
		\End_{\U(\mfB)_\NN}(\cH) = \End_{\U(\mfA')}(\cH) = \overline{\mfA},
	\]
	and the second identity follows from Theorem~\ref{thm:C*algebras-group}(1). This proves (1).
	
	Finally, by assumption we have $\overline{\mfB} = \mfA'$, and so we also conclude that $\overline{\mfA} = \mfA'' = (\overline{\mfB})' = \mfB'$. Hence, the last claim is proved.
\end{proof}

\begin{remark}\label{rmk:generates-commutant}
	The assumption $\overline{\mfB} = \mfA'$ considered in Proposition~\ref{prop:C*algebra-group-normtopo} will turn out to be fundamental in what follows. This condition can be phrased by saying that the von Neumann algebra generated by $\mfB$ is the commutant of $\mfA$. We note that, as seen from the arguments used in the proof of Proposition~\ref{prop:C*algebra-group-normtopo}, the identities $\overline{\mfB} = \mfA'$ and $\overline{\mfA} = \mfB'$ are in fact equivalent for any pair of $C^*$-subalgebras $\mfA$ and $\mfB$ of $\cB(\cH)$, where $\cH$ is any Hilbert space. In other words, $\mfB$ generates the commutant of $\mfA$ if and only if $\mfA$ generates the commutant of $\mfB$. Furthermore, in this case, $\overline{\mfA} \cap \overline{\mfB}$ is the center of both von Neumann algebras generated by $\mfA$ and $\mfB$.
\end{remark}

We introduce the following definition that subsumes the relationship just considered between two $C^*$-algebras.

\begin{definition}\label{def:commutant-dual}
	Let $\cH$ be a Hilbert space and let $\mfA, \mfB$ be two $C^*$-algebras with representations $\varphi, \psi$ on $\cH$, respectively. The $C^*$-algebras will be called commutant dual on $\cH$ if they satisfy the equivalent conditions $\overline{\psi(\mfB)} = \varphi(\mfA)'$ and $\overline{\varphi(\mfA)} = \psi(\mfB)'$. When the representations are clear from the context we will simply say that $\mfA$ and $\mfB$ are commutant dual. In particular, this applies when both are $C^*$-subalgebras of $\cB(\cH)$.
\end{definition}

\begin{remark}\label{rmk:commutant-dual-norm-separable}
	If $\cH$ is a separable Hilbert space and $\mfA$ is a $C^*$-subalgebra of $\cB(\cH)$, then there always exists a norm-separable $C^*$-subalgebra $\mfB$ of $\cB(\cH)$ which is commutant dual of $\mfA$. In the notation of Definition~\ref{def:commutant-dual}, this is a rephrasing of what was already noted in Remark~\ref{rmk:vonNeumann-decomposition} to be a consequence of \cite[Lemma~14.1.17]{KRvolII}.
\end{remark}

For a given $C^*$-subalgebra $\mfA$ of $\cB(\cH)$, where $\cH$ is a separable Hilbert space, Proposition~\ref{prop:C*algebra-group-normtopo} allows us to realize $\mfA$ as a strong-operator dense subalgebra of the intertwining operators for the unitary representation given by the Polish group of unitary elements of a norm-separable $C^*$-algebra $\mfB$. As noted in Remark~\ref{rmk:vonNeumann-decomposition}, the latter sort of algebra is the main ingredient needed to decompose von Neumann algebras with respect to direct integrals. This is how, in the next result, we realize any $C^*$-algebra acting on a separable Hilbert space as a (strong-operator) dense subset of the intertwining operators of a unitary representation that can be decomposed. 

\begin{theorem}\label{thm:C*algebra-intdecomp-group}
	Let $\cH$ be a separable Hilbert space and let $\mfA, \mfB$ be commutant dual $C^*$-subalgebras of $\cB(\cH)$ with $\mfB$ norm-separable. Let us assume that the central direct integral decomposition of $\cH$ associated to $\overline{\mfA}$ is given by the direct integral of $(\cH_p)_{p \in X}$ over $(X,\mu)$. Then, the unitary representation of $\U(\mfB)_\NN$ on $\cH$ can be decomposed with respect to this direct integral. More precisely, there exist a family $(\rho_p)_{p \in X}$ of maps that satisfy the following properties.
	\begin{enumerate}
		\item For every $p \in X$, the map $\rho_p : \U(\mfB)_\NN \rightarrow \U(\cH_p)$ is a unitary representation.
		\item For every $U \in \U(\mfB)_\NN$ we have $U(p) = \rho_p(U)$, for almost every $p \in X$.
	\end{enumerate}
	Furthermore, we can choose $(\rho_p)_{p \in X}$ such that the central decomposition of $\overline{\mfA}$ is given by $\big(\End_{\U(\mfB)_\NN}(\cH_p)\big)_{p \in X}$, and also such that $\rho_p = \varphi_p|_{\U(\mfB)_\NN}$, for almost every $p \in X$, where $(\varphi_p)_{p \in X}$ is a decomposition of $\mfB$.
\end{theorem}
\begin{proof}
	Let us consider a decomposition $(\varphi_p)_{p \in X}$ of $\mfB$ as given in Proposition~\ref{prop:C*normseparable-decomposable}. By Remark~\ref{rmk:decompositions-unit} there exists a null subset $N \subset X$ such that for every $p \in X$ the map $\varphi_p : \mfB \rightarrow \cB(\cH_p)$ is a $C^*$-algebra representation that takes the unit onto the unit. It follows that for every $p \in X \setminus N$, the map $\rho_p = \varphi_p|_{\U(\mfB)_\NN} : \U(\mfB)_\NN \rightarrow \U(\cH_p)_\NN$ is a continuous homomorphism of topological groups, where the target is endowed with the norm topology. After weakening the topology in $\U(\cH_p)_\NN$ to its strong-operator topology, the map $\rho_p$ remains continuous and so yields a unitary representation for every $p \in X \setminus N$. 
	
	If we define, for every $p \in N$, the map $\rho_p : \U(\mfB)_\NN \rightarrow \U(\cH_p)$ as the trivial unit representation, then the family $(\rho_p)_{p \in X}$ satisfies (1) and (2) by the previous paragraph. Note also that $\rho_p = \varphi_p|_{\U(\mfB)_\NN}$, for almost every $p \in X$, for the chosen decomposition $(\varphi_p)_{p \in X}$ of $\mfB$.
	
	On the other hand, for almost every $p \in X \setminus N$, the decomposition of the von Neumann algebra $\overline{\mfA}$ is given~by
	\[
		(\overline{\mfA})_p = ((\mfA')_p)' 
				= \big(\overline{\varphi_p(\mfB)}\big)' 
				= \rho_p(\U(\mfB)_\NN)'
				= \End_{\U(\mfB)_\NN}(\cH_p)
	\]
	where we have applied in the first identity \cite[Proposition~14.1.24]{KRvolII} to the (mutually commutant of each other) von Neumann algebras $\overline{\mfA}$ and $\mfA'$, that clearly contain the diagonalizable operators. The second identity follows from Definition~\ref{def:vonNeumann-decomposition} and the fact that $\mfA$ and $\mfB$ are commutant dual (see Definition~\ref{def:commutant-dual}). And the third identity follows from the previous paragraph. This completes the proof of the statement.
\end{proof}

\begin{remark}\label{rmk:C*algebra-intdecomp-group}
	The conclusions of Theorem~\ref{thm:C*algebra-intdecomp-group} yield a number of important improvements over those of Theorem~\ref{thm:C*algebras-group}, as well as for known results for unitary representations. For a $C^*$-subalgebra $\mfA$ of $\cB(\cH)$, once we choose a norm-separable commutant dual $\mfB$, we can replace the group $\U(\mfA')$ given in the latter with the subgroup $\U(\mfB)_\NN$ obtained in the former. This yields a new assignment
	\begin{equation}\label{eq:mfA-to-U(mfB)NN}
		\mfA \longmapsto \U(\mfB)_\NN,
	\end{equation}
	that allows us to work with a smaller subgroup whose unitary representation can actually be decomposed as stated in Theorem~\ref{thm:C*algebra-intdecomp-group}. This is achieved so that the ``component intertwining operators'' $\End_{\U(\mfB)_\NN}(\cH_p)$, for almost every $p \in X$, provide the decomposition of the von Neumann algebra $\overline{\mfA}$. In particular, although the above assignment does depend on the choice of $\mfB$, both the components of the representation of $\U(\mfB)_\NN$ and their intertwining operators yield the von Neumann algebra decompositions of $\mfA'$ and $\overline{\mfA}$, respectively. This claim follows from the proof of Theorem~\ref{thm:C*algebra-intdecomp-group}. Since these decompositions are unique (almost everywhere) the information provided by the group $\U(\mfB)_\NN$, as it relates to $\mfA$, is unique.
\end{remark}

As a consequence of the arguments used in the proof of Theorem~\ref{thm:C*algebra-intdecomp-group} we obtain the next result which has its own interest for representation theory.

\begin{corollary}\label{cor:group-decomposition}
	Let $\cH$ be a separable Hilbert space, let $G \subset \U(\cH)_\NN$ be a closed norm-separable subgroup with its natural unitary representation on $\cH$ and let $\mfR = G''$ be the von Neumann algebra generated by $G$. Let us assume that the central direct integral decomposition of $\cH$ associated to $\mfR$ is given by the family $(\cH_p)_{p \in X}$ over $(X,\mu)$. Then, there exists a family $(\rho_p)_{p \in X}$ of unitary representations $\rho_p : G \rightarrow \U(\cH_p)$, for every $p \in X$, such that the following properties hold.
	\begin{enumerate}
		\item The group $G$ consists of decomposable operators and $(\rho_p)_{p \in X}$ decomposes the representation of $G$ on $\cH$: for every $U \in G$ we have $U(p) = \rho_p(U)$, for almost every $p \in X$.
		\item The family $\big(\rho_p(G)''\big)_{p \in X}$ yields the decomposition of the von Neumann algebra $\mfR$.
	\end{enumerate}
\end{corollary}
\begin{proof}
	Let $\mfB$ denote the $C^*$-algebra generated by $G$. Then, $\mfB$ is the norm-closure of the linear space $\cV$ generated by $G$. Hence, the norm-separability of $G$ implies the same property for $\mfB$. Furthermore, we have $\mfR = \overline{\mfB}$.
	
	Following the arguments from the proof of Theorem~\ref{thm:C*algebra-intdecomp-group}, the norm separability of the $C^*$-algebra $\mfB$ allows us to realize $\cH$ as the direct integral of $(\cH_p)_{p \in X}$ over $(X,\mu)$ so that this yields precisely the central direct integral decomposition associated to $\mfR$. Such arguments allows us as well to obtain a decomposition $(\varphi_p)_{p \in X \setminus N}$ of $\mfB$, for some null subset $N$, such that the restrictions $\rho_p = \varphi_p|_G : G \rightarrow \U(\cH_p)$, where $p \in X \setminus N$, yield unitary representations for which conclusion (1) holds. 
	
	On the other hand, the decomposition of the von Neumann algebra $\mfR$ is given by choosing for each $p \in X \setminus N$ the von Neumann algebra
	\[
		\overline{\varphi_p(\mfB)} = \varphi_p(\mfB)''
			= \varphi_p(\cV)'' = \varphi_p(G)'' = \rho_p(G)'',
	\]
	where the second identity follows from the fact that $\mfB$ is the norm closure of $\cV$. Hence, conclusion (2) holds.
	
	Note that we can define $\rho_p$ arbitrarily for $p \in N$ without changing the previous claims. Hence, this completes the proof.
\end{proof}

\begin{remark}\label{rmk:group-decompositions}
	We observe that the decomposition of a unitary representation of a topological group $G$ as a direct integral of unitary representations is not at all trivial. It usually requires from $G$ to be a Lie group (see \cite{DixmierCAlgebras,MautnerProcAMS}), and even in this case the decomposition may not be uniquely associated to the group (see \cite{DixmierCAlgebras}). We have achieved such decomposition for the groups considered in Theorem~\ref{thm:C*algebra-intdecomp-group} and Corollary~\ref{cor:group-decomposition} which in general are not Lie groups. The possible lack of uniqueness with respect to the group is replaced by uniqueness with respect to the $C^*$-algebras and von Neumann algebras considered, as we have observed above. This is perfectly fine for our purposes, since our focus lies on $C^*$-algebras.
	
	On the other hand, we note that \cite[Theorem~1.2]{MautnerAnnMathI} is a related result that considers decompositions of arbitrary self-adjoint families of operators. However, the decompositions considered therein are restricted to some choices that do not allow in general to obtain the sort of decomposition of unitary representations that were obtained in Theorem~\ref{thm:C*algebra-intdecomp-group}.
\end{remark}

\subsection{Norm-separable and Abelian operator $C^*$-algebras}
As it has been highlighted in Theorem~\ref{thm:C*algebra-intdecomp-group}, for a given $C^*$-subalgebra $\mfA$ of $\cB(\cH)$, where $\cH$ is a separable Hilbert space, it is of interest to consider a norm-separable $C^*$-subalgebra $\mfB$ commutant dual of $\mfA$ (see Remark~\ref{rmk:generates-commutant} and Definition~\ref{def:commutant-dual}). We will now assume that $\mfA$ is norm-separable as well. In this case, we have two norm-separable commutant dual $C^*$-subalgebras $\mfA$ and $\mfB$ of $\cB(\cH)$. This yields an obvious symmetric setup. We also recall that, in such case, the common center of $\overline{\mfA}$ and $\overline{\mfB}$ is given by their intersection.

With the previous remarks in mind, we obtain the next result that allows to simultaneously decompose norm-separable commutant dual $C^*$-algebras. We recall the standing convention stated in Remark~\ref{rmk:decompositions-unit}.

\begin{proposition}\label{prop:C*algebra-mutualcomm-normseparable-intdecomp}
	Let $\cH$ be a separable Hilbert space and let $\mfA, \mfB$ be two norm-separable commutant dual $C^*$-subalgebras of $\cB(\cH)$. Let us assume that the central direct integral decomposition of $\cH$ associated to $\overline{\mfA}$ is given by the direct integral of $(\cH_p)_{p \in X}$ over $(X,\mu)$. Then, there exists two families $(\psi_p)_{p \in X}$ and $(\varphi_p)_{p \in X}$ of $C^*$-algebra representations of $\mfA$ and $\mfB$, respectively, that map $\psi_p : \mfA \rightarrow \cB(\cH_p)$ and $\varphi_p : \mfB \rightarrow \cB(\cH_p)$, for almost every $p \in X$, and that satisfy the following properties.
	\begin{enumerate}
		\item The families $(\psi_p)_{p \in X}$ and $(\varphi_p)_{p \in X}$ decompose the $C^*$-subalgebras $\mfA$ and $\mfB$, respectively: i.e.~we have for every $A \in \mfA$ and $B \in \mfB$ that $A(p) = \psi_p(A)$ and $B(p) = \varphi_p(B)$ holds for almost every $p \in X$.
		\item For almost every $p \in X$, the $C^*$-subalgebras $\psi_p(\mfA)$ and $\varphi_p(\mfB)$ of $\cB(\cH_p)$ are commutant dual: i.e.~we have $\overline{\psi_p(\mfA)} = \varphi_p(\mfB)'$ and $\overline{\varphi_p(\mfB)} = \psi_p(\mfA)'$, where the closures are taken in the strong-operator topology of $\cB(\cH_p)$.
	\end{enumerate}
\end{proposition}
\begin{proof}
	The existence of both families $(\psi_p)_{p \in}$ and $(\varphi_p)_{p \in X}$ satisfying (1) follows from Remark~\ref{rmk:vonNeumann-decomposition} since both $\mfA$ and $\mfB$ are norm-separable. It follows as well from this remark and Definition~\ref{def:vonNeumann-decomposition} that the assignments $p \mapsto \overline{\psi_p(\mfA)}$ and $p \mapsto \overline{\varphi_p(\mfB)}$ yield the central decompositions of $\overline{\mfA}$ and $\overline{\mfB}$, respectively.
	
	Since $\overline{\mfA}$ and $\overline{\mfB}$ are commutant of each other, an application of \cite[Proposition~14.1.24]{KRvolII} as in the proof of Theorem~\ref{thm:C*algebra-intdecomp-group} yields for almost every~$p \in X$
	\[
		\overline{\psi_p(\mfA)} = (\overline{\mfA})_p 
			= \big((\overline{\mfB})'\big)_p
			= \big((\overline{\mfB})_p\big)' = \big(\overline{\varphi_p(\mfB)}\big)'
			= \varphi_p(\mfB)',
	\]
	and so (2) follows, thus completing the proof.
\end{proof}

As a consequence we obtain the next result that improves the conclusions from Theorem~\ref{thm:C*algebra-intdecomp-group} for the case of norm-separable $C^*$-algebras.

\begin{theorem}\label{thm:C*algebra-normseparable-intdecomp-group}
	Let $\cH$ be a separable Hilbert space and $\mfA, \mfB$ be two norm-separable commutant dual $C^*$-subalgebras of $\cB(\cH)$. Let us assume that the central direct integral decomposition of $\cH$ associated to $\overline{\mfA}$ is given by the direct integral of $(\cH_p)_{p \in X}$ over $(X,\mu)$. Then, we have $\overline{\mfA} = \End_{\U(\mfB)_\NN}(\cH)$ and the representations of both $\mfA$ and $\U(\mfB)_\NN$ can be decomposed preserving this property. More precisely, there exist $(\psi_p)_{p \in X}$ and $(\rho_p)_{p \in X}$ families of $C^*$-algebra representations and unitary representations of $\mfA$ and $\U(\mfB)_\NN$, respectively, such that, for almost every $p \in X$, they map $\psi_p : \mfA \rightarrow \cB(\cH_p)$ and $\rho_p : \U(\mfB)_\NN \rightarrow \U(\cH_p)$ satisfying the following properties.
	\begin{enumerate}
		\item For every $A \in \mfA$ and $U \in \U(\mfB)_\NN$, we have $A(p) = \psi_p(A)$ and $U(p) = \rho_p(U)$, for almost every $p \in X$.
		\item We have $\overline{\psi_p(\mfA)} = \End_{\U(\mfB)_\NN}(\cH_p)$, for almost every $p \in X$.
	\end{enumerate}
\end{theorem}
\begin{proof}
	We note that the first claim follows from Proposition~\ref{prop:C*algebra-group-normtopo}(1). On the other hand, we recall again the convention introduced in Remark~\ref{rmk:decompositions-unit}.
	
	Let us consider families $(\psi_p)_{p \in X}$ and $(\varphi_p)_{p \in X}$ as obtained in Proposition~\ref{prop:C*algebra-mutualcomm-normseparable-intdecomp}. With the same arguments used in the proof of Theorem~\ref{thm:C*algebra-intdecomp-group} we consider the family of unitary representations $(\rho_p = \varphi_p|_{\U(\mfB)_\NN})_{p \in X}$, where this choice is applied almost everywhere and completed arbitrarily on some null set. Hence, (1) from our statement follows from Proposition~\ref{prop:C*algebra-mutualcomm-normseparable-intdecomp}(1).
	
	Finally, for almost every $p \in X$ we have
	\[
		\overline{\psi_p(\mfA)} = \varphi_p(\mfB)' = \rho_p(\U(\mfB)_\NN)'
			= \End_{\U(\mfB)_\NN}(\cH_p),
	\]
	by Proposition~\ref{prop:C*algebra-mutualcomm-normseparable-intdecomp}(2). This completes the proof.
\end{proof}

We now consider the case of Abelian norm-separable $C^*$-algebras. For this it will be useful to recall the next definition.

\begin{definition}\label{def:multifree}
	A unitary representation of a topological group $G$ on a Hilbert space $\cH$ is called multiplicity-free when the von Neumann algebra $\End_G(\cH)$ is Abelian.
\end{definition}

We now establish the decomposition into irreducible pieces of multiplicity-free unitary representations in the setup provided by our constructions.

\begin{corollary}\label{cor:normseparable-Abelian-intdecomp}
	Let $\cH$ be a separable Hilbert space and $\mfA, \mfB$ be two norm-separable commutant dual $C^*$-subalgebras of $\cB(\cH)$. Let us assume that $\mfA$ is Abelian and consider the direct integral decomposition of $\cH$ associated to $\overline{\mfA}$ as given by the direct integral of $(\cH_p)_{p \in X}$ over $(X,\mu)$. Then, the unitary representation of $\U(\mfB)_\NN$ on $\cH$ is multiplicity-free and for any pair of families $(\psi_p)_{p \in X}$ and $(\rho_p)_{p \in X}$ as in Theorem~\ref{thm:C*algebra-normseparable-intdecomp-group} the following is further satisfied for almost every $p \in X$
	\begin{enumerate}
		\item $\psi_p(\mfA) = \C I_p$, where $I_p$ is the identity operator on $\cH_p$, and
		\item the unitary representation $\rho_p : \U(\mfB)_\NN \rightarrow \U(\cH_p)$ is irreducible.
	\end{enumerate}
\end{corollary}
\begin{proof}
	Since $\overline{\mfA} = \End_{\U(\mfB)_\NN}(\cH)$ is Abelian, the unitary representation of $\U(\mfB)_\NN$ on $\cH$ is multiplicity-free. It also follows that for our choice of integral decomposition $\overline{\mfA}$ is precisely the von Neumann algebra of diagonalizable operators. Hence, the norm-separability of $\mfA$ and Theorem~\ref{thm:C*algebra-normseparable-intdecomp-group}(1) applied to the elements of a dense countable subset of $\mfA$ implies (1) from our statement for almost every $p \in X$. Once this is established, Theorem~\ref{thm:C*algebra-normseparable-intdecomp-group}(2) implies that $\End_{\U(\mfB)_\NN}(\cH_p) = \C I_p$ for almost every $p \in X$, and this yields the irreducibility of the unitary representations $\rho_p$. This proves (2). 
\end{proof}

\begin{remark}\label{rmk:normseparable-intdecomp}
	We proved in Theorem~\ref{thm:C*algebras-group} that any $C^*$-algebra $\mfA$ of operators acting on a Hilbert is strong-operator dense in the von Neumann algebra of intertwining operators of some unitary representation. Looking for more detailed descriptions our focus has been to study such an operator $C^*$-algebra $\mfA$ by choosing a commutant dual $C^*$-algebra. This leads us to Theorem~\ref{thm:C*algebra-intdecomp-group} where we proved that the group can be chosen nice enough allowing its unitary representation to be decomposed so that it provides a decomposition of $\overline{\mfA}$. Finally, we have proved in Theorem~\ref{thm:C*algebra-normseparable-intdecomp-group} that if $\mfA$ is norm-separable as well, then $\mfA$ itself can be decomposed so that its components are realized as intertwining operators for the components of the unitary representation. This provides a general setup to decompose all the operators involved: those in the given $C^*$-algebra $\mfA$ and those from the group with respect to which $\mfA$ consists of intertwining operators. Furthermore, this is achieved so that the same relation, intertwining operator vs action operator, is preserved in their components for the direct integral decomposition.
	
	If we further assume that $\mfA$ is Abelian, then Corollary~\ref{cor:normseparable-Abelian-intdecomp} provides additional information. The corresponding unitary representation is multiplicity-free and can be decomposed into irreducible representations with respect to the direct integral decomposition of the Hilbert space. Note that the group under consideration, $\U(\mfB)_\NN$ where $\mfB$ is a norm-separable $C^*$-algebra commutant dual of $\mfA$, is in general far from being a Lie group or even a locally compact group. Hence, the decomposition of the unitary representation of $\U(\mfB)_\N$ does not follow from the usual representation theory for the latter sort of groups.
\end{remark}

As a consequence, we now obtain an extension of Corollary~\ref{cor:group-decomposition} for multiplicity-free representations. Again, this result has an interest of its own.

\begin{corollary}\label{cor:group-decomposition-multfree}
	Let $\cH$ be a separable Hilbert space and let $G \subset \U(\cH)_\NN$ be a closed norm-separable subgroup whose representation on $\cH$ is multiplicity-free. Let us also denote by $\mfR = G''$ the von Neumann algebra generated by $G$. If the central direct integral decomposition of $\cH$ associated to $\mfR$ is given by $(\cH_p)_{p \in X}$ over $(X,\mu)$, then there exists a family $(\rho_p)_{p \in X}$ of unitary representations $\rho_p : G \rightarrow \U(\cH_p)$ that decomposes $G$ into irreducible components. More precisely, the following properties hold.
	\begin{enumerate}
		\item The group $G$ consists of decomposable operators, and for every $U \in G$ we have $U(p) = \rho_p(U)$, for almost every $p \in X$.
		\item For almost every $p \in X$, the representation $\rho_p : G \rightarrow \U(\cH_p)$ is irreducible.
	\end{enumerate}
\end{corollary}
\begin{proof}
	Let us consider $\mfB$ the $C^*$-subalgebra of $\cB(\cH)$ generated by $G$. In particular, and as noted before, $\mfB$ is norm-separable. Choose $\mfA$ a norm-separable $C^*$-algebra commutant dual of $\mfB$. Hence, we have
	\[
	\overline{\mfA} = \mfB' = G' = \End_G(\cH),
	\]
	where the second identity follows from our choice of $\mfB$ in terms of $G$. Since $G$ has a multiplicity-free representation on $\cH$, we conclude that $\mfA$ is Abelian and so that the hypothesis from Corollary~\ref{cor:normseparable-Abelian-intdecomp} hold. An application of the latter to our current setup yields a family of unitary representations $(\rho_p)_{p \in X}$ that satisfy (1) from our statement since $G \subset \U(\mfB)_\NN$.
	
	Let us now consider the families $(\psi_p)_{p \in X}$ and $(\varphi_p)_{p \in X}$ of $C^*$-algebra representations of $\mfA$ and $\mfB$, respectively, from the proof of Theorem~\ref{thm:C*algebra-normseparable-intdecomp-group}. In particular, we have $\rho_p = \varphi_p|_{\U(\mfB)_\NN}$, for almost every $p \in X$. Hence, we obtain
	\[
		\End_G(\cH_p) = \rho_p(G)' 
			= \varphi_p(\mfB)' = \overline{\psi_p(\mfA)}
			= \overline{\C I_p} = \C I_p,
	\]
	for almost every $p \in X$. The second identity follows from the choice of $\mfB$ in terms of $G$. The third identity follows from Proposition~\ref{prop:C*algebra-mutualcomm-normseparable-intdecomp}(2). And the fourth identity is a consequence of Corollary~\ref{cor:normseparable-Abelian-intdecomp}(1). The end conclusion is that $\rho_p : G \rightarrow \U(\cH_p)$ is irreducible for almost every $p \in X$. This completes the proof.
\end{proof}

We now consider the case of $C^*$-algebras generating maximal Abelian von Neumann algebras. We will use the next known fact (see \cite[Theorem~14.5]{ConwayOT} for a similar result) whose proof we include for the sake of completeness. Our proof is based on the previous constructions.

\begin{lemma}\label{lem:maximalAbelian-vonNeumann}
	Let $\cH$ be a separable Hilbert space and $\mfR$ be a von Neumann algebra acting on $\cH$. Then, $\mfR$ is a maximal Abelian subalgebra of $\cB(\cH)$ if and only if there exists a locally compact Polish space $X$ with a Radon measure $\mu$ and a unitary map $U : \cH \rightarrow L^2(X,\mu)$ such that $U \mfR U^* = L^\infty(X,\mu)$, where the latter is identified with the corresponding von Neumann algebra of multiplier operators.
\end{lemma}
\begin{proof}
	For $(X,\mu)$ as in the statement, it is an elementary fact that $L^\infty(X,\mu)$ is a maximal von Neumann algebra for its action by multiplier operators on $L^2(X,\mu)$ (see \cite[Example~5.1.6]{KRvolI}). This implies the sufficiency in our statement.
	
	Conversely, for $\mfR$ maximal Abelian let us consider the direct integral decomposition of $\cH$ associated to $\mfR$ and given by $(\cH_p)_{p \in X}$ over $(X,\mu)$. The maximality property of $\mfR$ implies that $\mfR = \mfR'$. In particular, the algebras of decomposable and diagonalizable operators coincide. Since the choice $\C I_p$, for almost every $p \in X$, where $I_p$ is the identity operator on $\cH_p$, yields a decomposition of the latter, then \cite[Theorem~14.1.24]{KRvolII} implies that the choice $\cB(\cH_p)$, for almost every $p \in X$, yields a decomposition of the former. The uniqueness of the decomposition of von Neumann algebras together with $\mfR = \mfR'$ implies that $\dim \cH_p = 1$, for almost every $p \in X$. Hence, the direct integral decomposition of $\cH$ yields a unitary operator $U : \cH \rightarrow L^2(X,\mu)$ which clearly satisfies the required properties.
\end{proof}

The next result characterizes and describes norm-separable $C^*$-algebras that generate maximal Abelian von Neumann algebras.

\begin{theorem}\label{thm:maximal-vN-group}
	Let $\cH$ be a separable Hilbert space and $\mfA$ be an Abelian norm-separable $C^*$-subalgebra of $\cB(\cH)$. Then, the following conditions are equivalent.
	\begin{enumerate}
		\item $\overline{\mfA}$ is a maximal Abelian subalgebra of $\cB(\cH)$.
		\item $\mfA$ is commutant dual of itself: i.e.~we have $\overline{\mfA} = \mfA'$.
		\item $\overline{\U(\mfA)_\NN}$ is a maximal Abelian subgroup of $\U(\cH)$, where the closure is taken in $\U(\cH)$ with respect to its (strong-operator) topology.
		\item There exist a locally compact Polish space $X$ with a Radon measure $\mu$ and a unitary map $U : \cH \rightarrow L^2(X,\mu)$ for which we have $U\overline{\mfA}U^* = L^\infty(X,\mu)$, where the latter is identified with the corresponding von Neumann algebra of multiplier operators.
		\item If the direct integral decomposition of $\cH$ associated to $\overline{\mfA}$ is given by the direct integral of $(\cH_p)_{p \in X}$ over $(X,\mu)$, then $\mu(\{p \in X \mid \dim \cH_p > 1\}) = 0$.
	\end{enumerate}
	In this case, we can choose $(X,\mu)$ such that $\mu$ is positive on non-empty open subsets of $X$ and also such that, with respect to the unitary map $U$ in (4), we have 
	\[
		\overline{\mfA} = \End_{\U(\mfA)_\NN}(\cH) \simeq 
			\End_{C_\infty(X,\T)}(L^2(X,\mu)), 
	\]
	where $C_\infty(X,\T) = \{ f \in C_\infty(X) \mid f(X) \subset \T \}$ is the (Polish) group of unitary elements of the norm-separable $C^*$-algebra $C_\infty(X) = \{ f \in C(X) \mid \lim_{p \to \infty} f(p) \text{ exists} \}$, and where the unitary representation of $C_\infty(X,\T)$ on $L^2(X,\mu)$ is given by multiplier operators.
\end{theorem}
\begin{proof}
	The equivalence of (1) and (4) follows from Lemma~\ref{lem:maximalAbelian-vonNeumann}. We also note that $\overline{\mfA}$ is maximal Abelian if and only if $\overline{\mfA} = \big(\overline{\mfA}\big)' = \mfA'$, which proves the equivalence of (1) and (2).
	
	We now assume that (1) holds and proceed to prove (3). Let $G$ be an Abelian closed subgroup of $\U(\cH)$ that contains $\overline{\U(\mfA)_\NN}$. Let $\mfR = G''$ be the von Neumann algebra generated by $G$, which is Abelian. The latter follows from Kaplansky Density Theorem and the strong-operator continuity of the product on bounded sets. Hence, we have $\mfR \supset \big(\U(\mfA)_\NN\big)'' = \overline{\mfA}$ and the maximality of $\overline{\mfA}$ implies that $\mfR = \overline{\mfA}$. We thus obtain
	\[
		G \subset \U(\mfR) = \U(\overline{\mfA}) = \overline{\U(\mfA)_\NN},
	\]
	where the second identity follows from Proposition~\ref{prop:KaplanskyUnitaryDensity}. This implies $G = \overline{\U(\mfA)_\NN}$ and shows that (3) holds.
	
	Let us now assume (3) and prove that (1) holds. Let $\mfR$ be an Abelian von Neumann algebra containing $\overline{\mfA}$, which clearly implies that $\U(\mfR)$ is an Abelian closed subgroup of $\U(\cH)$. Hence, we have 
	\[
		\overline{\U(\mfA)_\NN} = \U(\overline{\mfA}) \subset \U(\mfR),
	\]
	where we have applied Proposition~\ref{prop:KaplanskyUnitaryDensity} again. Thus, the maximality of $\overline{\U(\mfA)_\NN}$ implies that the equality in the last inclusion. We conclude that $\overline{\mfA} = \mfR$, and this proves that (1) holds.
	
	If we assume (5), then we clearly have $\cH \simeq L^2(X,\mu)$ through some unitary map and in such a way that $\overline{\mfA}$ corresponds to $L^\infty(X,\mu)$. Hence, (5) implies (4). 
	
	Let us now assume that (2) holds, and consider the direct integral decomposition of $\cH$ associated to $\overline{\mfA}$ as given by the direct integral of $(\cH_p)_{p \in X}$ over $(X,\mu)$. Then, for almost every $p \in X$ we have
	\[
		\C I_p = \big(\overline{\mfA}\big)_p = (\mfA')_p = \cB(\cH_p).
	\]
	The first identity follows from the fact that $\overline{\mfA}$ is the von Neumann algebra of diagonalizable operators. For the third identity, we apply \cite[Theorem~14.2.4]{KRvolII} using the fact that $\overline{\mfA}$ is maximal Abelian in $\mfA'$ since they are in fact equal. This proves that $\dim \cH_p = 1$ for almost every $p \in X$, and so that (2) implies (5).	
	
	Let us assume that the equivalent conditions from the statement hold and let us choose a unitary map $U : \cH \rightarrow L^\infty(X,\mu)$ so that all properties stated in (4) are satisfied. Note that by Remark~\ref{rmk:support-Radon} we can choose $X$ and $\mu$ so that the latter is positive on the non-empty open subsets of the former.
	
	We observe that the identity $\overline{\mfA} = \End_{\U(\mfA)_\NN}(\cH)$ follows from Proposition~\ref{prop:C*algebra-group-normtopo}(1) by taking $\mfB = \mfA$ in its statement, which can be done since $\mfA$ is norm-separable and commutant dual of itself.
	
	With respect to the correspondence given by $U$ we observe that Theorem~\ref{thm:C*algebras-group}(1) implies the identity in the following expression
	\[
		\overline{\mfA} \simeq L^\infty(X,\mu) = \End_{\U(L^\infty(X,\mu))}(L^2(X,\mu)),
	\]
	where we have used that the von Neumann algebra of multiplier operators given by $L^\infty(X,\mu)$ is its own commutant. On the other hand, we clearly have $\U(L^\infty(X,\mu)) = L^\infty(X,\T)$ where the latter denotes the group of functions $f \in L^\infty(X,\mu)$ whose essential range lies in $\T$.	Note that this is considered as a Polish group of multiplier operators endowed with the strong-operator topology inherited from $\U(L^2(X,\mu))$.
	
	We claim that $C_c(X)$ is strong-operator dense in $L^\infty(X,\mu)$ where both are considered as algebras of multiplier operators acting on $L^2(X,\mu)$. Let $f \in L^\infty(X,\mu)$, $g \in L^2(X,\mu)$ and $\epsilon > 0$ be given. It follows from topological and measure theoretic arguments that there exists an open subset $W \subset X$ with finite measure such~that
	\[
		\int_{X \setminus W} |fg|^2 \dif \mu < \epsilon.
	\]
	By a well known consequence of Lusin's Theorem there exists $(h_n)_n \subset C_c(W)$ that converges to $f$ almost everywhere in $W$ and such that $\|h_n\|_\infty \leq \|f\|_\infty$. We can extend the functions $h_n$ by $0$ on $X \setminus W$ to obtain a sequence of functions that belong to $C_c(X)$ and that satisfy the same properties. We will denote such sequence with the same symbol. Hence, we have $\lim_{n \to \infty} |fg - h_ng|^2 = 0$ almost everywhere in $W$ and $|fg - h_ng|^2 \leq 2\|f\|_\infty^2 |g|^2$ almost everywhere in $X$, so that the dominated convergence theorem implies that there exists $n_0$ for which
	\[
		\int_X |fg - h_n g|^2 \dif \mu = \int_{X \setminus W} |fg|^2 \dif \mu	
				+ \int_W |fg - h_ng|^2 \dif \mu < 2\epsilon,
	\]
	for every $n \geq n_0$. This proves that the multiplier algebra of $C_c(X)$ is indeed strong-operator dense in the one obtained from $L^\infty(X,\mu)$. Hence, the same strong-operator density in $L^\infty(X,\mu)$ holds for the multiplier algebra given by the $C^*$-algebra $C_\infty(X)$. By Lemma~\ref{lem:Cb-into-BL2}, the realization of the latter as a multiplier algebra is an isomorphism of $C^*$-algebras onto its image. This follows from our choice above of $\mu$ so that it is positive on non-empty open subsets of $X$. In particular, $C_\infty(X)$ can be considered as a $C^*$-subalgebra of $\cB(L^2(X,\mu))$. Since the group of unitary elements of $C_\infty(X)$ is $C_\infty(X,\T)$, it follows from Proposition~\ref{prop:KaplanskyUnitaryDensity} that the latter is dense in the topological group $L^\infty(X,\T)$ with respect to the strong-operator topology. Such density and the claims proved so far clearly imply that
	\[
		\overline{\mfA} \simeq \End_{L^\infty(X,\T)}(L^2(X,\mu))
			= \End_{C_\infty(X,\T)}(L^2(X,\mu)),
	\]
	where the identification is given by $U$. 
	
	We complete the proof by showing that $C_\infty(X)$ is norm-separable. Since $X$ is locally compact and Polish it follows that $\widehat{X}$, its one-point compactification, is metrizable (see Corollary in \cite[Page~158]{BourbakiGenTop5-10}). Hence, Stone-Weierstrass Theorem implies that $C(\widehat{X})$ is a separable $C^*$-algebra (see \cite[Remark~3.4.15]{KRvolI}). Since the $C^*$-algebras $C_\infty(X)$ and $C(\widehat{X})$ are naturally isomorphic we conclude that the former is separable as well.
\end{proof}

Theorem~\ref{thm:maximal-vN-group} allows us to state a correspondence between $C^*$-algebras and groups when both satisfy an Abelian maximality condition. In the next result, $\mathrm{MASOT}$ stands for maximal Abelian in (closure with respect to the) strong-operator topology.

\begin{corollary}\label{cor:maximal-vN-group}
	For a separable Hilbert space $\cH$, let us consider the families
	\begin{align*}
		\mathrm{MASOT}(\cB(\cH)) 
			&= \bigg\{ \mfA \subset \cB(\cH) \Big|
			\begin{matrix}
				\mfA \text{ is a norm-separable } \\
					\;C^*\text{-subalgebra and $\overline{\mfA}$ is maximal Abelian } 
			\end{matrix} \bigg\}, \\
		\mathrm{MASOT}(\U(\cH)) &= \bigg\{ G \subset \U(\cH) \Big|
			\begin{matrix}
				G \text{ is a norm-separable norm-closed } \\ 
					\text{ subgroup and $\overline{G}$ is maximal Abelian  }
			\end{matrix} \bigg\},
	\end{align*}
	and the assignments given by
	\begin{align*}
		\mathrm{MASOT}(\cB(\cH)) &\longrightarrow 	
				\mathrm{MASOT}(\U(\cH)) &
		\mathrm{MASOT}(\U(\cH)) &\longrightarrow 	
				\mathrm{MASOT}(\cB(\cH)) \\
		\mfA &\longmapsto \U(\mfA)_\NN &
		G &\longmapsto \mfA_G,
	\end{align*}
	where $\mfA_G$ is the $C^*$-algebra generated by $G$ and $\overline{G}$ is the (strong-operator) closure of $G$ in $\U(\cH)$. Then, these assignments satisfy the following identities.
	\begin{enumerate}
		\item For every $\mfA \in \mathrm{MASOT}(\cB(\cH))$: $\mfA_{\U(\mfA)_\NN} = \mfA$.
		\item For every $G \in \mathrm{MASOT}(\U(\cH))$: $\overline{G} = \overline{\U(\mfA_G)_\NN}$.
	\end{enumerate}
	In other words, up to strong-operator closure in $\U(\cH)$ for $\mathrm{MASOT}(\U(\cH))$, the above correspondences are inverses of each other.
\end{corollary}
\begin{proof}
	It follows from Theorem~\ref{thm:maximal-vN-group} that the first correspondence is well defined. Furthermore, for every $\mfA \in \mathrm{MASOT}(\cB(\cH))$ we clearly have $\mfA_{\U(\mfA)_\NN} = \mfA$. In particular, property (1) holds.
	
	On the other hand, if $G \in \mathrm{MASOT}(\U(\cH))$ and $\mfR$ is an Abelian von Neumann algebra containing $\overline{\mfA_G} = G''$, then we have
	\[
		\overline{G} \subset \overline{\U(\mfA_G)_\NN} = \U(\overline{\mfA_G})
			\subset \U(\mfR),
	\]
	where we have used Proposition~\ref{prop:KaplanskyUnitaryDensity}. The maximality property of $\overline{G}$ implies that $\overline{G} = \U(\mfR)$, from which we obtain
	\[
		\mfR = \U(\mfR)'' = \overline{G}'' = G'' = \overline{\mfA_G},
	\]
	thus showing that $\overline{\mfA_G}$ is maximal Abelian. Hence, the second assignment is well defined.
	
	Finally, for $G \in \mathrm{MASOT}(\U(\cH))$ the group $\overline{\U(\mfA_G)_\NN}$ is Abelian and since it contains $G$ we conclude that $\overline{G} = \overline{\U(\mfA_G)_\NN}$. This proves (2) and completes the proof.
\end{proof}

One can apply the results obtained so far to formulate a description, using representation theory of Polish groups, of Abelian $C^*$-subalgebras of a given norm-separable $C^*$-algebra. We start with the following result, which adds to Theorem~\ref{thm:C*algebra-normseparable-intdecomp-group} and Corollary~\ref{cor:normseparable-Abelian-intdecomp} a condition on the Polish group involved in terms of the ambient $C^*$-algebra.

\begin{corollary}
\label{cor:Abelian-subalgebras-cT}
	Let $\cH$ be a separable Hilbert space and let $\cT$ be a norm-separable $C^*$-subalgebra of $\cB(\cH)$. Then, for every Abelian $C^*$-subalgebra $\mfA$ of $\cT$ the group $G = \U(\mfB)_\NN$ (where $\mfB$ is a norm-separable $C^*$-algebra commutant dual of $\mfA$) from Corollary~\ref{cor:normseparable-Abelian-intdecomp} further satisfies $\End_G(\cH) \subset \overline{\cT}$ or, equivalently,
	$G'' \supset \cT'$.
\end{corollary}
\begin{proof}
	By our assumptions $\End_G(\cH) = \overline{\mfA} \subset \overline{\cT}$, which is equivalent to $G'' = \overline{\mfA}' = \mfA' \supset \cT'$.
\end{proof}

The next result reduces the construction of Abelian subalgebras of a given norm-separable operator $C^*$-algebra to Polish subgroups of $\U(\cH)_\NN$ with multiplicity-free representation on $\cH$. In this result, the notation $\mathrm{MultFree}$ stands for multiplicity-free Polish groups.

\begin{theorem}\label{thm:Abelian-subalg-vs-multiplicity-free}
	Let $\cH$ be a separable Hilbert space and let $\cT$ be a norm-separable $C^*$-subalgebra of $\cB(\cH)$. Let us also consider the family
	\[
		\mathrm{MultFree}(\cT) =
			\Bigg\{ G \subset \U(\cH)_\NN \;\bigg|
				\begin{matrix}
					\;\End_G(\cH) \subset \overline{\cT} \\
					\;G \text{ is a closed subgroup, } \\
					\text{ $G \rightarrow \U(\cH)$ is 
						multiplicity-free }
				\end{matrix} \Bigg\},
	\]
	Then, the family of Abelian $C^*$-subalgebras of $\cT$ can be described as follows.
	\begin{itemize}
		\item Let $\mfA \subset \cT$ be a $C^*$-subalgebra. Then, $\mfA$ is Abelian if and only if there exists a group $G \in \mathrm{MultFree}(\cT)$ such that $\mfA \subset \End_G(\cH) \cap \cT$. In particular, if $\mfA$ is a maximal Abelian $C^*$-subalgebra of $\cT$, then there exits $G \in \mathrm{MultFree}(\cT)$ such that $\mfA = \End_G(\cH) \cap \cT$.
	\end{itemize}
\end{theorem}
\begin{proof}
	Let us consider $G$ a Polish group that belongs to the family $\mathrm{MultFree}(\cT)$. By definition of the latter, $\End_G(\cH)$ is Abelian and so every $C^*$-subalgebra of the $C^*$-algebra $\End_G(\cH) \cap \cT$ is clearly an Abelian $C^*$-subalgebra of $\cT$.
	
	Conversely, let $\mfA \subset \cT$ be an Abelian $C^*$-subalgebra. Then, Corollary~\ref{cor:Abelian-subalgebras-cT} yields a Polish closed subgroup $G \subset \U(\cH)_\NN$ satisfying the required conditions. The claim on maximal Abelian $C^*$-subalgebras is now obvious.
\end{proof}

\begin{remark}\label{rmk:Abelian-subalg-vs-multiplicity-free}
	Theorem~\ref{thm:Abelian-subalg-vs-multiplicity-free} provides a tool to find Abelian $C^*$-subalgebras in a given ambient norm-separable $C^*$-algebra $\cT$. The key is to look for Polish closed subgroups $G \subset \U(\cH)_\NN$, with multiplicity-free representation on $\cH$, and consider the von Neumann algebra $\End_G(\cH)$ as well as its intersection with $\cT$. This is very much obvious from the constructions so far. The main simplification obtained from Theorem~\ref{thm:Abelian-subalg-vs-multiplicity-free} and Corollary~\ref{cor:Abelian-subalgebras-cT} is given by the equivalent conditions $G'' \supset \cT'$ and $\End_G(\cH) \subset \overline{\cT}$. These say that for smaller $\cT$ we have to consider less Polish groups $G$. Conversely, for larger $\cT$ we have to consider more Polish groups $G$. In particular, for $\cT$ strong-operator dense in $\cB(\cH)$ we will have, in general, to consider all closed subgroups of $\U(\cH)_\NN$.
\end{remark}

\subsection{Compact groups vs Polish groups}
Let us consider a compact group $K$ with a unitary representation on a separable Hilbert space $\cH$. Hence, we have a so-called isotypic decomposition associated to the $K$-action written as
\[
	\cH = \bigoplus_{j \in J} \cH_j,
\]
and characterized by the following properties.
\begin{enumerate}
	\item The Hilbert direct sum is $K$-invariant, in other words, for each $j \in J$, the closed subspace $\cH_j$ is $K$-invariant.
	\item For every $j \in J$, $\cH_j \not= 0$ and it is the sum of all irreducible $K$-submodules of $\cH$ isomorphic to some fixed $K$-module $V_j$.
	\item For $j_1 \not= j_2$ in $J$, the $K$-modules $V_{j_1}$ and $V_{j_2}$ are non-isomorphic.
\end{enumerate}
Since $\cH$ is assumed to be separable, it follows that $J$ is countable. We also note that, with this setup, the unitary representation of $K$ on $\cH$ is multiplicity-free if and only if $\cH_j$ is an irreducible $K$-module for every $j \in J$. Another standard property is that of finite multiplicity that, for the current setup and notation, means that $\cH_j$ can be written a finite sum of irreducible $K$-modules isomorphic to $V_j$, for all $j \in J$.

Let us denote by $\rho : K \rightarrow \U(\cH)$ the unitary representation considered and by $\rho_j : K \rightarrow \U(\cH_j)$, where $j \in J$, the corresponding representation on $\cH_j$. We recall that every irreducible representation of a compact group is finite dimensional. For this reason, the representation of $K$ has finite multiplicity if and only if $\dim \cH_j < \infty$, for all $j \in J$. For simplicity, we will assume that the representation $\rho$ has finite multiplicity.

The von Neumann algebra of intertwining operators for the unitary representation $\rho$ satisfies
\[
	\End_K(\cH) = \bigoplus_{j \in J} \End_K(\cH_j)
			\simeq \bigoplus_{j \in J} M_{m_j}(\C),
\]
where, with the above notation, $\cH_j$ is the direct sum of $m_j$ copies of $V_j$, and $M_{m_j}(\C)$ denotes the $C^*$-algebra of $m_j \times m_j$ complex matrices. This fact is an easy consequence of Schur's Lemma.

From the previous constructions, the main goal is to consider strong-operator dense norm-separable $C^*$-subalgebras $\mfA$ in $\End_K(\cH)$, but using some Polish closed subgroup $G \subset \U(\cH)_\NN$. The latter, so that we can apply the machinery built up to this point. Hence, we are looking for such a subgroup $G$ so that  $\End_K(\cH) = \End_G(\cH)$. We recall that a choice of such a group was $G = \U(\mfB)_\NN$, with $\mfB$ a norm-separable $C^*$-algebra commutant dual of $\mfA$. This amounts to finding such a $\mfB$ strong-operator dense in
\[
	\mfA' = \overline{\mfA}' = \rho(K)'',
\]
the von Neumann algebra generated by $\rho(K)$. This image $\rho(K)$ consists of operators of the form
\[
	\bigoplus_{j \in J} \rho_j(g),
\]
where $g \in K$. However, from the elementary properties of von Neumann algebras, we have
\[
	\rho(K)'' = \End_K(\cH)' 
		= \bigoplus_{j \in J} \End_K(\cH_j)'
		= \bigoplus_{j \in J} \rho_j(K)'',
\]
which is much larger and not norm-separable in general. However, we can explicitly exhibit a strong-operator dense norm-separable $C^*$-subalgebra in the next result. Recall that $J$ is countable and, with our current assumptions, if it is finite, then $\cH$ is finite dimensional. Since the latter case is rather trivial, we will assume from now on that $J = \N$. We will also denote by $I_j$ the identity operator on $\cH_j$, for every $j \in \N$.

\begin{lemma}\label{lem:mfB-for-K''}
	With the previous notation and assumptions, the space defined by
	\[
		\mfB = \bigg\{
				T = \bigoplus_{j = 0}^\infty T_j 
					\in \bigoplus_{j = 0}^\infty \rho_j(K)''
					\;\Big|\; 
					\lim_{j \mapsto \infty} \|T_j - I_j\| = 0
			\bigg\}
	\]
	is a strong-operator dense norm-separable $C^*$-subalgebra of $\rho(K)''$.
\end{lemma}
\begin{proof}
	First, it is easy to see that $\mfB$ is a $*$-subalgebra of $\rho(K)''$. Let us consider a sequence $(T_{(k)})_{k \in \N} \subset \mfB$ with elements of the form
	\[
		T_{(k)} = \bigoplus_{j = 0}^\infty T_{k,j}
	\]
	which is norm-convergent to $T \in \cB(\cH)$. In particular, we necessarily have
	\[
		T = \bigoplus_{j = 0}^\infty T_j
	\]
	with $T_j \in \rho_j(K)''$, for all $j \in \N$. By considering the estimate
	\[
		\|T_j - I_j\| \leq \|T_j - T_{k,j}\| + \|T_{k,j} - I_j\|
			\leq \|T - T_{(k)}\| + \|T_{k,j} - I_j\|,
	\]
	we conclude that $\lim_{j \mapsto \infty} \|T_j - I_j\| = 0$. Hence, $\mfB$ is a $C^*$-subalgebra. 
	
	On the other hand, it is easy to see that
	\[
		\mfB'' = \bigoplus_{j = 0}^\infty \rho_j(K)'' = \rho(K)'',
	\]
	and so the strong-operator density of $\mfB$ in $\rho(K)''$ is a consequence of the Double Commutant Theorem.
	
	Finally, the norm-separability of $\mfB$ follows by considering the subset
	\[
		\bigg\{
			T = \bigoplus_{j = 0}^\infty T_j 
			\in \mfB
			\;\Big|\; 
				\begin{matrix}
					\text{ $T_j \not= I_j$ for finitely many $j$ } \\
					\text{ $T_j \in S_j$ for all $j$ }
				\end{matrix}
			\bigg\}
	\]
	where $S_j \subset \rho_j(K)''$ is a countable dense subset for all $j \in \N$. Recall that $\rho_j(K)'' \subset M_{m_j}(\C)$, for all $j$. It is straightforward to prove that the latter set is norm-dense in $\mfB$.
\end{proof}

As a consequence, we have a description of $\End_K(\cH)$ in terms of the unitary representation of a Polish closed subgroup of $\U(\cH)$.

\begin{theorem}\label{thm:compact-vs-Polish}
	Let $\rho : K \rightarrow \U(\cH)$ be a unitary representation of a compact group $K$ on the separable Hilbert space $\cH$. Let us assume that $\rho$ has an isotypic decomposition with finite multiplicity
	\[
		\cH = \bigoplus_{j = 0}^\infty \cH_j,
	\]
	and let $\rho_j : K \rightarrow \U(\cH_j)$, for $j \in \N$, be the corresponding components of $\rho$. Then, for $\mfB$ as in Lemma~\ref{lem:mfB-for-K''}, the group $G = \U(\mfB)_\NN$ is a Polish closed subgroup of $\U(\cH)_\NN$ such that $\End_K(\cH) = \End_G(\cH)$. Furthermore, if $\rho$ is multiplicity-free, then
	\[
		G = \bigg\{
			T = \bigoplus_{j = 0}^\infty T_j 
				\in \bigoplus_{j = 0}^\infty \U(n_j)
				\;\Big|\; 
				\lim_{j \mapsto \infty} \|T_j - I_j\| = 0
			\bigg\},
	\]
	where $n_j = \dim \cH_j$, for all $j \in \N$.
\end{theorem}
\begin{proof}
	That $G$ is a Polish closed subgroup of $\U(\cH)_\NN$ follows from Lemma~\ref{lem:mfB-for-K''}. We also conclude from the latter that $\overline{\mfB} = \End_K(\cH)'$ and so Proposition~\ref{prop:C*algebra-group-normtopo} implies that $\End_K(\cH) = \End_G(\cH)$.
	
	Let us now assume that $\rho$ is multiplicity-free and let us fix $j \in \N$. Then, $\rho_j$ is irreducible, and so we have $\rho_j(K)' = \C I_j$ which implies $\rho_j(K)'' = M_{n_j}(\C)$. The unitary elements of the latter are precisely $\U(n_j)$. This yields the required description for the group $G$.
\end{proof}

\begin{remark}\label{rmk:compact-vs-Polish}
	As it is well known, there are plenty of compact groups $K$ admitting multiplicity-free unitary representations $\rho$ on infinite dimensional separable Hilbert spaces $\cH$. If we apply the previous discussion to such examples, we immediately observe that the Polish group $G$ obtained from Theorem~\ref{thm:compact-vs-Polish} is quite larger than the image $\rho(K)$. Nevertheless, there is something useful about the former. The operators that belong to the group $G$ allow to consider all the terms in the isotypic decomposition independently, instead of the way the components of the operators in $\rho(K)$ are tied together. This can be done while still having at our disposal a direct sum decomposition for $G$.
	
	On the other hand, it is possible to make a smaller choice of the Polish group as the following result shows. Its proof uses the same arguments as above. We observe that the group $G$ obtained carries all the representations of $K$ involved but in an independent way.
\end{remark}

\begin{corollary}\label{cor:compact-vs-Polish}
	With the notation from Theorem~\ref{thm:compact-vs-Polish}, let us assume that $\rho$ is multiplicity-free. Then, the group
	\[
		\widehat{K} = \bigg\{
			T = \bigoplus_{j = 0}^\infty T_j 
				\in \bigoplus_{j = 0}^\infty \rho_j(K)
				\;\Big|\; 
				\lim_{j \mapsto \infty} \|T_j - I_j\| = 0
			\bigg\},
	\]
	is a Polish closed subgroup of $\U(\cH)_\NN$ that satisfies $\End_{\widehat{K}}(\cH) = \End_K(\cH)$.
\end{corollary}

\subsection*{Acknowledgement}
This research was partially supported by SNI-Conacyt and Conacyt Grants 280732 and 61517.

\end{document}